\RequirePackage{fix-cm}
\documentclass[final,12pt]{elsarticlehack}
\usepackage{amsmath,amsfonts,amsthm,amssymb}
\usepackage{placeins}
\usepackage{graphicx}
\usepackage{enumitem,xcolor}
\usepackage{pifont}
\usepackage{bm}
\usepackage[ruled,vlined,linesnumbered,norelsize]{algorithm2e}
\usepackage[hypertexnames=false,colorlinks=true,breaklinks=true,bookmarks=true,urlcolor=blue,citecolor=blue,linkcolor=blue,bookmarksopen=false,draft=false]{hyperref}
\usepackage{color}

\makeatletter
\renewenvironment{proof}[1][\proofname]{\par
  \pushQED{\qed}%
  \normalfont \topsep6\p@\@plus6\p@\relax
  \trivlist
  \item[\hskip\labelsep
        \itshape
    #1]\ignorespaces
}{%
  \popQED\endtrivlist\@endpefalse
}
\makeatother

\DeclareMathOperator{\conv}{conv}

\DeclareMathOperator*{\argmax}{argmax}
\DeclareMathOperator*{\argmin}{argmin}
\newcommand{\Po}{\mathcal{P}}
\newcommand{\D}{\mathcal{D}}

\makeatletter
\renewcommand*\env@matrix[1][c]{\hskip -\arraycolsep
  \let\@ifnextchar\new@ifnextchar
  \array{*\c@MaxMatrixCols #1}}
\makeatother

\newtheorem{thm}{Theorem}{\bf}{\rm}
\newtheorem{cor}[thm]{Corollary}{\bf}{\rm}
{\bf}{\rm}
\newtheorem{prop}[thm]{Proposition}{\bf}{\rm}
\newdefinition{defn}[thm]{Definition}{\bf}{\rm}
\newdefinition{rem}[thm]{Remark}{\bf}{\rm}
\newdefinition{ex}[thm]{Example}{\bf}{\rm}
 
\makeatletter
\renewcommand*{\top}{%
  {\mathpalette\@transpose{}}%
}
\newcommand*{\@transpose}[2]{%
  \raisebox{\depth}{$\m@th#1\scriptscriptstyle\mathsf{T}$}%
}
\makeatother

\allowdisplaybreaks[1]

\begin{document}
\begin{frontmatter}
\title{On disjunction convex hulls by lifting\tnoteref{label0}}
\tnotetext[label0]{A short preliminary version of this work appeared 
in the proceedings of ISCO 2024; see \cite{QL_ISCO2024}}

\author[label1]{Yushan Qu}
\author[label1]{Jon Lee}

\affiliation[label1]{organization={University of Michigan},
            city={Ann Arbor},
            state={MI},
            country={USA}}

\begin{abstract}
We study the natural extended-variable formulation for the
disjunction of $n+1$ polytopes in $\mathbb{R}^d$. We demonstrate that the convex hull 
$\D$ in  the natural extended-variable space $\mathbb{R}^{d+n}$ is given by full optimal big-M lifting 
(i) when $d\leq 2$ (and that it is not generally true for $d\geq 3$), and also (ii) under some technical conditions, when the polytopes  have a common facet-describing constraint matrix, for arbitrary $d\geq 1$ and $n\geq 1$. 
We give a broad family of examples
with $d\geq 3$ and $n=1$, where the convex hull is not described after employing all full optimal big-M lifting inequalities, but it is 
described after one round of MIR inequalities. 
Additionally, we give some general results on the polyhedral structure of $\D$,
and we demonstrate that all facets of 
$\D$ can be enumerated in polynomial time when $d$ is fixed.
\end{abstract}


\end{frontmatter}

\section*{Introduction}
In the context of mathematical-optimization models, the so-called ``big-M'' method is a 
classical way to treat disjunctions 
by using binary variables (see \cite[Section 26-3.I, parts (b--d,g)]{Dantzig}, \cite[Sections 8.1.2--3]{LeeLP}, for example). 
Perhaps it is the most powerful, feared, and 
abused modeling technique in mixed-integer optimization. 
For a simple example, we may have the choice
(for $x\in \mathbb{R}$) of $x=0$
or $\bm{\ell}\leq x\leq \bm{u}$ (where the constants satisfy
$0<\bm{\ell} < \bm{u}$).  In applications, for  $x$ in the ``operating range'' $[\bm{\ell},\bm{u}]$, we might incur a cost $\bm{c}>0$, 
which is commonly modeled with a binary variable $z$, and cost term $\bm{c}z$ (in a minimization objective function). It remains to link $x$ and $z$ with some constraints. 
Carelessly, we could use the continuous relaxation:
\begin{align}
&0 \leq x \leq \bm{u},\nonumber\\
&  (z-m_2) \frac{\bm{\ell}}{1-m_2} \leq x \leq \frac{\bm{u}}{1-m_1}z,\label{eq:force}
\end{align}
for any $0\leq m_i <1$, for $i=1,2$. 
Note that for $z=0$, \eqref{eq:force} reduces to $ -\frac{m_2}{1-m_2}\bm{\ell} \leq  x \leq 0$,
the left-hand inequality being redundant with respect to $x\geq 0$, and altogether implying that $x=0$.
And for $z=1$,  \eqref{eq:force} reduces to 
$\bm{\ell} \leq x \leq \frac{\bm{u}}{1-m_1}$, the right-hand inequality of which is redundant with respect to $x\leq \bm{u}$,  altogether implying that $\bm{\ell}\leq x \leq \bm{u}$.
So the model is valid --- and not  generally recommended, for $m_i>0$.

In this simple example, we can easily check
that we get the tightest relaxation possible (the so-called ``convex-hull relaxation'')
if and only if we  
choose $m_1=m_2=0$, whereupon  
 \eqref{eq:force} becomes $\bm{\ell} z \leq x \leq \bm{u} z$.
Notice that if we take $m_1$ (resp., $m_2$) to be near 1,
then the coefficient on $z$ on the right-hand (resp., left-hand)
inequality of  \eqref{eq:force}
becomes big --- a so-called ``big-M''. 
In the context of mixed-integer \emph{nonlinear} optimization, 
the story of this little example does not even end here 
(see, for example, \cite{lee_gaining_2020,PLPerspecCTW,lee2020piecewise}, and the references therein). 

Although our example is quite simple and perhaps trivial, 
the phenomenon is serious for modelers and solvers. 
From \cite{GurobiM},
we have: ``Big-M constraints are a regular source of instability for optimization problems.''
But it is not just numerical instability;  
from \cite{GurobiM}, we also have: 
``In other words, $x$ can take a [very small] positive value without incurring an expensive fixed charge on [$z$], which subverts the intent of only allowing a non-zero value for $x$ when the binary variable [$z$] has the value of 1.''
This behavior means weak continuous relaxations and
often very poor behavior of algorithms like branch-and-bound, branch-and-cut, etc., on such models. 

In higher dimension $d$ for $x$ (above it was $d=1$) and higher number $n+1$ (above it was $n+1=2$) of polytope regions  $\Po_i$ for $x$, the 
situation is cloudier than what was exposed in the simple example above, and 
our high-level goal is to 
carefully investigate when we
can nicely characterize the convex $\D$ in the natural extended-variable space, i.e., in $\mathbb{R}^{d+n}$ (see \S\ref{sec:D_def} for a precise definition of $\D$).

\medskip
\noindent {\bf Our main contributions and organization}:
In \S\ref{sec:def}, we lay out our main definitions, in particular ``full optimal big-M lifting''.
In \S\ref{sec:general}, we present
general polyhedral results for $\D$  (Theorems \ref{lem:fulldim}-\ref{lem:sumz}).
In \S\ref{sec:lowdim}, we
demonstrate that
for $d\in\{1,2\}$, with $n\geq 1$ arbitrary, $\D$ is completely described by employing full optimal big-M lifting inequalities (Theorem \ref{thm:hull}).
Also, we will see that this is not true for $d=3$. 
In \S\ref{sec:allfacets}, 
we demonstrate that all facets of 
$\D$ can be enumerated in polynomial time,
when $d$ is fixed (Theorem \ref{thm:computing}).
In \S\ref{sec:standard},
motivated by ideas in \S\S\ref{sec:lowdim}--\ref{sec:allfacets}, 
we  give a broad family of examples
with $d\geq 3$ and $n=1$, where $\D$ is not described after employing all full optimal big-M lifting inequalities (Theorem \ref{thm:allfacetssimplex}), but it is 
described after one round of MIR (mixed-integer rounding) inequalities (Theorem \ref{thm:MIRsimplex}). 
In \S\ref{sec:common}, 
we demonstrate that 
$\D$ is completely described by full optimal big-M lifting
when the polytopes have a common 
facet-describing constraint matrix 
(for arbitrary $d\geq 1$ and $n\geq 1$), under an  important technical condition
(Theorem \ref{thm:hull3}).
 We note that the cases
 covered by both Theorems \ref{thm:hull} and \ref{thm:hull3}
 are very natural for applications and within spatial branch-and-bound (in particular, the second case is already useful when the polytopes
are all hyper-rectangles; see Corollary \ref{cor:hyper}), the main algorithmic paradigm for 
mixed-integer nonlinear optimization. \S\ref{sec:outlook} contains some directions for further study.

\medskip
\noindent
{\bf Standard terminology concerning polytopes.} We follow the terminology  
of \cite[Chapter 0.5]{Lee_Cambridge}.
If $\alpha^\top x \leq \beta$ is satisfied by
all points of a polytope $\Po\subset \mathbb{R}^n$, then it is \emph{valid} for $\Po$. For a valid inequality $\alpha^\top x \leq \beta$  of $\Po$, the \emph{face described} is $\Po\cap \{ x \in \mathbb{R}^n ~:~ \alpha^\top x = \beta\}$.
The \emph{dimension} of a polytope $\Po$ is one less than the maximum number of affinely independent points in $\Po$. A polytope $\Po\subset \mathbb{R}^n$ is \emph{full dimensional} if its dimension is $n$. A \emph{facet} of $\Po$ is a face of dimension one less than that of $\Po$. If $\alpha^\top x \leq \beta$
describes a facet of a full-dimensional polytope $\Po$, then any other
inequality that also describes that face is a positive multiple of $\alpha^\top x \leq \beta$. Additionally, for a full-dimensional polytope $\Po$, the solution set of its (essentially unique) facet-describing inequalities is precisely $\Po$. Also, see \cite{Grunbaum2003} and \cite{ziegler} for
much more on polytopes.

\section{Definitions}\label{sec:def}

\subsection{Model and notation}
 In what follows, for $n\geq 1$ we let $N_0 := \{0,1,...,n\}$ and  $N:= \{1,...,n\}$. 
 For $d\geq 1$ and $n\geq 1$, we consider full-dimensional polytopes $\Po_i:=\{ x\in\mathbb{R}^d ~:~ \bm{A}^i x \leq \bm{b}^i\}$, for  $i\in N_0$\,. 
The number of columns of each matrix $\bm{A}^i$ is $d$, and the number of rows of each $\bm{A}^i$ agrees with the number of elements of the corresponding vector $\bm{b}^i$\,. We note that we do not assume that the $\Po_i$ are pairwise disjoint,
but that could be the case, especially for $d=1$. 
For convenience, we assume that each system $\bm{A}^i x\leq \bm{b}^i$ has
 no redundant inequalities, and then because of the full-dimensional assumption, we have that each inequality defining $\Po_i$ describes a unique facet of $\Po_i$\,. 
 
We define $n$ binary variables $z_i$\,,
with $z_i=1$ indicating that the vector of variables $x$ must be in $\Po_i$\,, for $i \in N$. Further, if $z_i=0$ for all $i \in N$, then the vector of variables $x$ must be in $\Po_0$\,. 
Finally, we constrain the $z_i$ via
\begin{align}
\textstyle \sum_{i\in N} z_i \leq 1, \label{eq:choice}
\end{align}
so, overall, 
$x$ must be in at least one $\Po_i$\,, for $i\in N_0$. Because the $\Po_i$ may overlap, it can be that $z_i=0$ but $x\in \Po_i$ for some $i\in N$. Likewise, we can have $z_i=1$ for 
 some choice of $i\in N$ but $x\in \Po_0$\,.

\subsection{Convex hull}\label{sec:D_def}
We write $\genfrac(){0pt}{1}{x}{z}\in\mathbb{R}^{d+n}$, 
where
$x\in\mathbb{R}^d$ and $z\in\mathbb{R}^n$.
For $i\in N$,  let $\mathbf{e}_i$ denote the $i$-th standard unit vector in $\mathbb{R}^n$,
and additionally for convenience, let  $\mathbf{e}_0$
denote the zero vector in $\mathbb{R}^n$.
For $i\in N_0$\,, let $\bar{\Po}_i:=\{ \genfrac(){0pt}{1}{x}{\mathbf{e}_i}~:~ x\in \Po_i\}$,
which is the polytope $\Po_i\subset \mathbb{R}^d$ lifted
to $\mathbb{R}^{d+n}$ by setting $z=\mathbf{e}_i$\,.
For $i\in N_0$\,, let $\mathcal{X}_i$ be the (finite) set of extreme points of $\Po_i$\,, 
and let $\bar{\mathcal{X}}_i:=\{\genfrac(){0pt}{1}{x}{\mathbf{e}_i} ~:~ x\in \mathcal{X}_i\}$, the (finite) set of extreme points of 
$\bar{\Po}_i$\,. 
 
Let
\begin{align*}
\D:=& \conv
\left\{
\bar{\Po}_i ~:~ i\in N_0
\right\}
= \conv
\left\{
\bar{\mathcal{X}}_i ~:~ i\in N_0
\right\}.
\end{align*}
The definition of $\D$ is an ``inner description'',
but convenient algorithmic approaches generally
work with an ``outer description'' (i.e., 
via linear inequalities). 
See Figures \ref{fig:ex} 
and  \ref{fig:ex2} for examples.
In the right-hand plot of Figure
\ref{fig:ex}, the bottom 
and top facets are described by $0\leq z_1 \leq 1$. 
In Figure \ref{fig:ex2},
the three (trapezoidal) facets are 
    described by $z_1\geq 0$, $z_2\geq 0$, and $z_1+z_2\leq 1$, while the  front and back (triangular) facets are the liftings of the lower and upper bounds for the three intervals.

\begin{figure}[ht]
\centering
\includegraphics[width=0.48\textwidth]{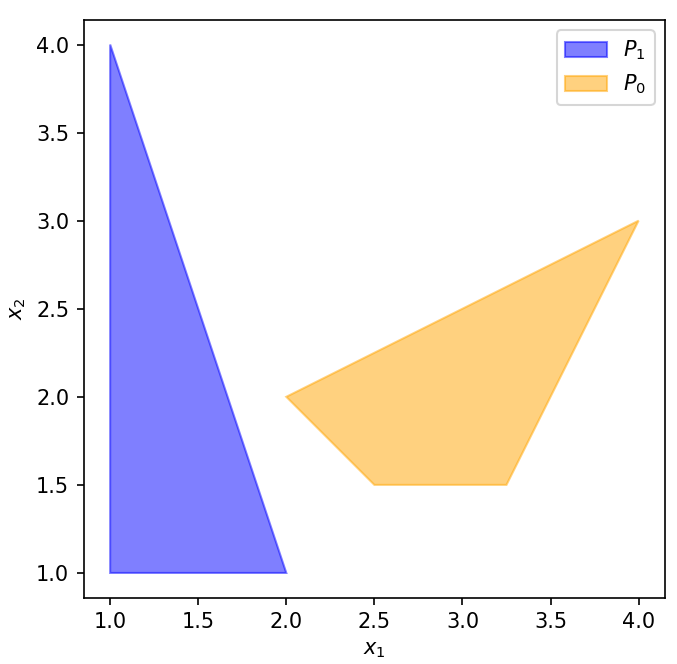}
~
\includegraphics[width=0.48\textwidth]{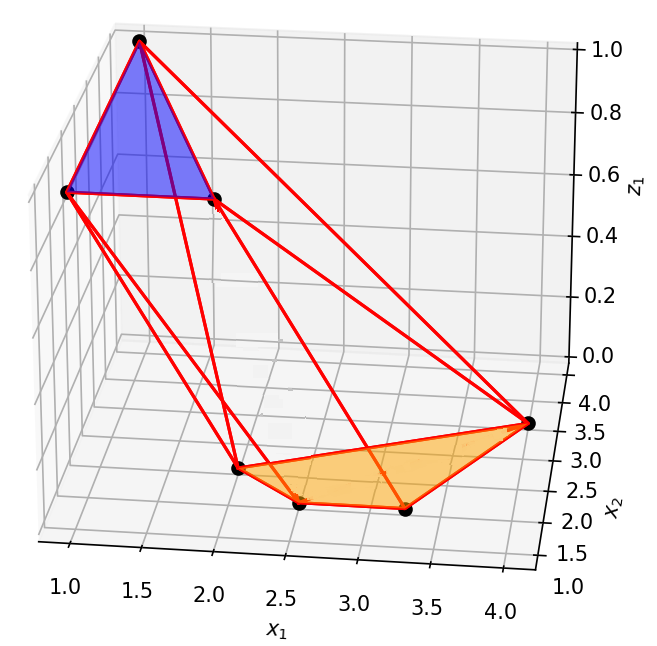}
\caption{Example: $d=2$, $n=1$}\label{fig:ex}
\end{figure}

\begin{figure}[ht]
\centering
\includegraphics[width=0.80\textwidth]
{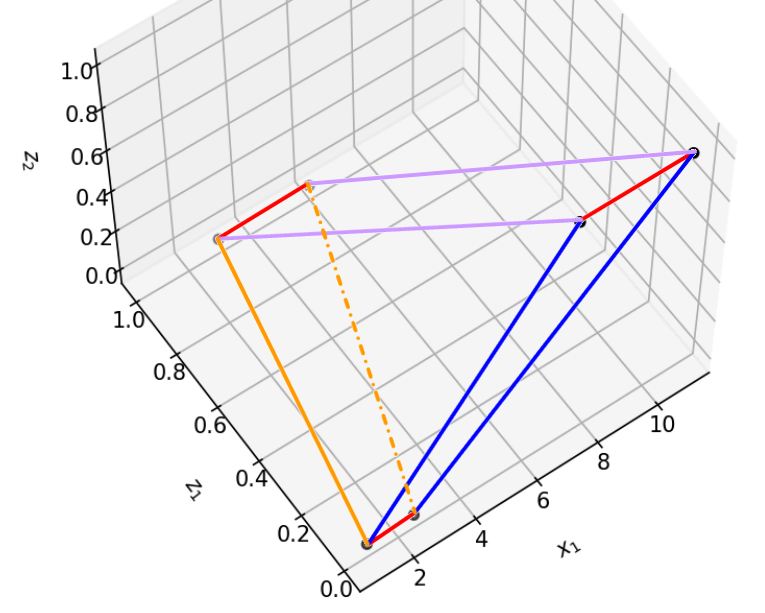}
\caption{Example: $d=1$, $n=2$}\label{fig:ex2}
\end{figure}

\FloatBarrier

\subsection{Full optimal big-M lifting}
Lifting is a general well-known technique
for extending linear inequalities in disjunctive settings.
See, for example, \cite[Chapter II.2, Section 1]{NWbook}.
Next, we describe how to do optimal lifting
of valid inequalities for individual $\Po_i$ to account for the other polytopes.

\bigskip

\noindent\underline{Case 1}: 
Suppose that 
\begin{equation*}
\alpha^\top x \leq \beta \tag{$*$}\label{eq:star}
\end{equation*}
is a valid inequality for $\Po_0$\,.
Now, for some $j\in N$,   we want to
optimally 
 lift $M_j z_j$ into 
this inequality, by choosing the largest $M_j$
so that the resulting inequality
\begin{equation*}
\alpha^\top x + M_j z_j \leq \beta 
\end{equation*}
is valid also for 
 $\bar{\Po}_j$\,.
 We clearly need to choose 
\begin{equation*}
    M_j \leq \beta -\alpha^\top x ,
\end{equation*}
for \emph{all} $x\in \Po_j$\,. And so, to get the
best such lifting, we set
\begin{equation*}
M_j := \min \{ \beta -\alpha^\top x ~:~ x \in \Po_j\}.
\end{equation*}
It is easy to see that starting with any valid linear inequality for $\Po_0$\,, and lifting  one coefficient at a time on the $z_j$\,, until we have lifted all of them, the order of lifting does not matter (because $\sum_{j\in N} z_j \leq 1$ is required) --- we will arrive at the same ``full-lifting inequality'' for $\D$, having the form
 \begin{equation} \textstyle
\alpha^\top x + \sum_{i\in N} M_i z_i \leq \beta. \tag{$*_0$}\label{eq:star0}
\end{equation}

\smallskip

\noindent\underline{Case 2}:
The above lifting development assumes that we start with the
inequality 
\eqref{eq:star}  being  valid  for $\Po_0$\,.
But instead, let us assume that we start with \eqref{eq:star}
being valid for $\Po_k$\,, for some $k\in N$. 
To calculate the optimal lifting coefficient for some $j\in N\setminus\{k\}$, we proceed exactly as above. But to lift the inequality so that is valid for  $\bar{\Po}_0$, we need to proceed a bit differently. We need to choose $M_0$ so that 
 the resulting inequality
 \begin{equation*} \textstyle
\alpha^\top x + M_0 \left( 1- \sum_{i\in N} z_i\right) \leq \beta 
\end{equation*}
is valid also for 
 $\bar{\Po}_0$\,.
   We clearly need to choose 
\begin{equation*}
    M_0 \leq \beta -\alpha^\top x ,
\end{equation*}
for \emph{all} $x\in \Po_0$\,.
And so, to get the
best such lifting, we set
\begin{equation*}
M_0 := \min \{ \beta -\alpha^\top x ~:~ x \in \Po_0\}.
\end{equation*}
Again for this case, it is easy to see that starting with any valid linear inequality for some $\Po_k$\,, and lifting  one  at a time  until we have lifted all, the order of lifting does not matter  --- we will arrive at the same inequality for $\D$. 
In this situation, gathering terms together,
the resulting ``full-lifting inequality''
is 
 \begin{equation} \textstyle
\alpha^\top x + \sum_{i\in N \setminus\{k\}} (M_i-M_0) z_i -M_0 z_k\leq \beta  - M_0\,.
\tag{$*_k$}\label{eq:stark}
\end{equation}

In both cases, it is clear that for each 
facet-defining inequality of each $\Po_k$\,, 
we can compute all of its 
optimal ``big-M'' coefficients simultaneously,
to lift to an inequality valid for $\D$. 
In a nod to the fact that in our setting,
the sequence of lifting does not matter,
and thus the same coefficients are
in a modeling sense optimal big-M
coefficients, we refer to \eqref{eq:star0}
and \eqref{eq:stark} as
\emph{full optimal big-M liftings} of \eqref{eq:star}.

\begin{ex}\label{ex:d_equal_1_lifting}
Suppose that $d=1$ and that $\Po_i:=[\bm{\ell}_i\,,\, \bm{u}_i]$, for $i\in N_0$\,. Consider the valid inequality for $\Po_0$: $x\leq \bm{u}_0$\,. In this case, if we lift all $z_i$ variables, the 
full optimal big-M lifting inequality is
$x+\sum_{i=1}^n M_i z_i \leq \bm{u}_0$\,, where
$M_i=\min\{ \bm{u}_0-x ~:~ x\in [\bm{\ell}_i\,,\, \bm{u}_i]\}=\bm{u}_0-\bm{u}_i$\,.
That is, we obtain: 
$x+\sum_{i=1}^n (\bm{u}_0-\bm{u}_i) z_i \leq \bm{u}_0$\,.
Similarly, if we take the valid inequality for $\Po_0$: $-x_0 \leq - \bm{\ell}_0$\,, then 
$M_i=\min\{ -\bm{\ell}_0+x ~:~ x\in [\bm{\ell}_i\,,\, \bm{u}_i]\}=\bm{\ell}_i-\bm{\ell}_0$\,.
That is, we obtain the full optimal big-M lifting inequality
$x+\sum_{i=1}^n (\bm{\ell}_0-\bm{\ell}_i) z_i \geq \bm{\ell}_0$\,.
\end{ex}

Of course disjunctive formulations and the associated liftings that 
we describe and analyze are not new at all.
Disjunctive programming, at a high level, is the subject of concrete convexification of nonconvex sets, 
in the context of global optimization.
This important topic was introduced by Balas in 1971; see \cite{BalasDP}.
It is a key technique in the well-known integer-programming book
\cite{NWbook}. Fairly recently, four decades after its introduction, and with lots of progress in the intervening time
period,  Balas himself, in probably his last major work, finally
organized the topic into a major text; see \cite{balas2018disjunctive}.
In closely related work, \cite{Jeroslow,Blair,BALAS1988279} studied when the convex hull of polytopes has a simple form,
but they do not work with 
binary variables to manage
a disjunction.
We prefer to include binary variables because of their modeling power (e.g., carrying fixed costs
and logical relationships).
\cite[Section 2.11]{CCZ_Book}
treats disjunctions, but 
discusses models in a higher-dimensional space, with copies of the variables $x$ introduced
for each $\Po_i$\,.
\cite[Section 6]{Vielma2015MixedIL} is closer 
to our viewpoint, but we have opted to present a self-contained treatment
which we hope is broadly accessible to mixed-integer optimization users.

Finally, there is a lot of experience and expectedly many enhancements for methods
for handling disjunctions (see, e.g.,
\cite{belotti2011disjunctive,Lodi-et-al:2020,TRESPALACIOS201598,Misener_Disj}), and with  various particularizations
(see \cite{kilincc2015two,audet}, for example).

Still, there is a lot to understand about the elementary situation that we have set out. In particular, in what follows, 
we will investigate conditions for which the
convex hull $\D$ is completely described by full optimal big-M lifting of the facet-describing inequalities of the $\Po_i$\,.

\section{General polyhedral results}\label{sec:general}

First we present elementary valid inequalities for $\D$ and establish (under mild technical conditions) that they describe facets of $\D$. 

\begin{thm}
\label{lem:fulldim}
Suppose that $d\geq 1$. 
Then if all $\Po_i$ are nonempty, for $i\in N_0$\,, and at least one is full dimensional, then 
$\D$ is full dimensional.
\end{thm}

\begin{thm}
\label{lem:facet_lift}
Suppose that $d\geq 1$. 
Then if all $\Po_i$ are nonempty, for $i\in N_0$\,,  then 
the full optimal big-M lifting of a facet-describing inequality of a full-dimensional $\Po_k$ (for any $k\in N_0$)
is facet describing for $\D$. 
\end{thm}

\begin{thm}
\label{lem:nonneg}
Suppose that $d\geq 1$. 
Then if all $\Po_i$ are nonempty, for $i\in N_0$\,, and $\Po_0$ is full dimensional, then for all $j\in N$, $z_j\geq 0$ is facet describing for $\D$. 
\end{thm}

\begin{thm}
\label{lem:sumz}
Suppose that $d\geq 1$. 
Then if all $\Po_i$ are nonempty, for $i\in N_0$\,,
and $\Po_k$ is full dimensional for some $k\in N$, then $\sum_{j\in N} z_j \leq 1$
is facet describing for $\D$. 
\end{thm}

\begin{proof}~\!\!\!\!{\it [of Theorem \ref{lem:fulldim}].}
There are two cases to consider.
First, we assume that $\Po_0$ is full dimensional.
Because $\Po_0$ is full dimensional, 
then $\Po_0$ contains a set of $d+1$ affinely-independent points, say $\hat{x}^0,\hat{x}^1,\ldots,\hat{x}^d$.
Additionally, we assumed that
each $\Po_i$\,, for $i \in N$, is nonempty, so it contains a point, say $\tilde{x}^i$.
Now, we consider the $n+d+1$ points (in $\mathbb{R}^{n+d})$: 
\[
\genfrac(){0pt}{1}{\hat{x}^0}{\mathbf{e}_0}
,
\genfrac(){0pt}{1}{\hat{x}^1}{\mathbf{e}_0}
,\ldots,
\genfrac(){0pt}{1}{\hat{x}^d}{\mathbf{e}_0}
;
\genfrac(){0pt}{1}{\tilde{x}^1}{\mathbf{e}_1}
,
\genfrac(){0pt}{1}{\tilde{x}^2}{\mathbf{e}_2}
,\ldots,
\genfrac(){0pt}{1}{\tilde{x}^n}{\mathbf{e}_n},
\]
and it remains only to show that these  $n+d+1$ points are affinely independent.
 
We append a row of $1$'s (conveniently) in the middle of the matrix and construct the following square matrix
\begin{equation*}
   S_0 := 
    \left[
\begin{array}{cccc|cccc}
        \hat{x}^0    & \hat{x}^1    & \cdots & \hat{x}^d    & \tilde{x}^1    & \tilde{x}^2    & \cdots & \tilde{x}^n    \\
        1           &  1           & \cdots & 1            & 1            & 1            &        \cdots & 1            \\
        \hline 
        0           &  0           & \cdots & 0            & 1            & 0            &        \cdots & 0 \\
        0           &  0           & \cdots & 0            & 0            & 1            &        \cdots & 0 \\
        .           &  .           & \cdots & .            & .            & .            &        \cdots & . \\
        0           &  0           & \cdots & 0            & 0            & 0            &        \cdots & 1 \\
\end{array}
    \right]. 
\end{equation*}

Letting $C := 0_{n\times(d+1)}$\,,     $D := I_n$\,,
\[
    A := 
    \begin{bmatrix}
        \hat{x}^0    & \hat{x}^1    & \cdots & \hat{x}^d \\
        1           &  1           & \cdots & 1 
    \end{bmatrix}_{(d+1)\times(d+1)}\,, \mbox{ and }
    B := 
    \begin{bmatrix}
    \tilde{x}^1    & \tilde{x}^2    & \cdots & \tilde{x}^n    \\
    1 & 1            &        \cdots & 1            \\
    \end{bmatrix}_{(d+1)\times n}\,,
\]
we have, $S_0 = \begin{bmatrix}
        A & B \\
        C & D\\
    \end{bmatrix} 
    =  \begin{bmatrix}
        A & B \\
        0 & I\\
    \end{bmatrix} $,
    and therefore $\det S_0 = \det(A)$,
    which is nonzero because $\hat{x}^0,\hat{x}^1,\ldots,\hat{x}^d$ are affinely independent.

For the other case, we assume instead that  $\Po_k$ is full dimensional for some $k\in N$.
Because $\Po_k$ is full dimensional, 
it contains a set of $d+1$ affinely-independent points, say $\hat{x}^0,\hat{x}^1,\ldots,\hat{x}^d$.
Additionally, because each $\Po_i$ is nonempty for $i\in N_0\setminus\{k\}$,
it contains a point, say $\tilde{x}^i$.
Now, we consider the $n+d+1$ points (in $\mathbb{R}^{n+d})$: 
\[
\genfrac(){0pt}{1}{\hat{x}^0}{\mathbf{e}_k}
,
\genfrac(){0pt}{1}{\hat{x}^1}{\mathbf{e}_k}
,\ldots,
\genfrac(){0pt}{1}{\hat{x}^d}{\mathbf{e}_k}
;
\genfrac(){0pt}{1}{\tilde{x}^0}{\mathbf{e}_0}
,
\genfrac(){0pt}{1}{\tilde{x}^1}{\mathbf{e}_1}
,\ldots,
\genfrac(){0pt}{1}{\tilde{x}^{k-1}}{\mathbf{e}_{k-1}}
,
\genfrac(){0pt}{1}{\tilde{x}^{k+1}}{\mathbf{e}_{k+1}}
,\ldots,
\genfrac(){0pt}{1}{\tilde{x}^n}{\mathbf{e}_n},
\]
and it remains only to show that these  $n+d+1$ points are affinely independent
 
 We append a row of $1$'s (conveniently) in the middle of the matrix and construct the following square matrix
\begin{equation*}
     \left[
\begin{array}{cccc|ccccccc}
        \hat{x}^0    & \hat{x}^1    & \cdots    & \hat{x}^d    & \tilde{x}^0    & \tilde{x}^1    & \cdots    & \tilde{x}^{k-1}    & \tilde{x}^{k+1}    & \cdots    & \tilde{x}^n \\
        1    &  1    & \cdots & 1    & 1    & 1    & \cdots    & 1    & 1    & \cdots    & 1\\
        \hline
        \mathbf{e}_k    & \mathbf{e}_k    & \cdots    & \mathbf{e}_k    & \mathbf{e}_0    & \mathbf{e}_1    & \cdots    & \mathbf{e}_{k-1}    & \mathbf{e}_{k+1}    & \cdots    & \mathbf{e}_n \\
\end{array}
    \right].
\end{equation*}
Subtracting row $d+1$ from row $d+k+1$, preserving the determinant,
we arrive at the matrix 
$    
\begin{bmatrix}
        A & B \\
        0 & D\\
    \end{bmatrix} $,
where 
\begin{align*}
&A := \begin{bmatrix}
        \hat{x}^0    & \hat{x}^1    & \cdots    & \hat{x}^d \\
        1    &  1    & \cdots & 1 \\
    \end{bmatrix},
    B := \begin{bmatrix}
        \tilde{x}^0    & \tilde{x}^1    & \cdots    & \tilde{x}^{k-1}    & \tilde{x}^{k+1}    & \cdots    & \tilde{x}^n \\
        1    & 1    & \cdots    & 1    & 1    & \cdots    & 1\\
    \end{bmatrix},\\
&    \mbox{ and }
    D := \begin{bmatrix}
        \mathbf{e}_0-\mathbf{e}_k\,,    & \mathbf{e}_1-\mathbf{e}_k\,,    & \cdots\,,    & \mathbf{e}_{k-1}-\mathbf{e}_k\,,    & \mathbf{e}_{k+1}-\mathbf{e}_k\,,    & \cdots\,,    & \mathbf{e}_n-\mathbf{e}_k
    \end{bmatrix}.
\end{align*}
Because again $\det A \not=0$,
it remains only to demonstrate that $\det D \not=0$.
But it is easy to see that multiplying the first column by -1, adding that to the remaining columns, and then moving the first column immediately after the $(k-1)$-st column,
turns $D$ into an identity matrix.
\end{proof}

\begin{proof}~\!\!\!\!{\it [of Theorem \ref{lem:facet_lift}].}
There are two cases to consider.

\noindent\underline{Case 1}: First, we assume that $\Po_0$ is full dimensional. 
{\color{black} Let 
\eqref{eq:star}
be a facet-describing inequality of $\Po_0$\,. Let  
\eqref{eq:star0}
be its full lifting, where $M_i := \min \{ \beta -\alpha^\top x ~:~ x \in \Po_i\}$.}
We need to choose $n+d$ affinely-independent points from $\D$ satisfying \eqref{eq:star0} as an equation to demonstrate that  \eqref{eq:star0} describes a facet of $\D$. 
Let $\hat{x}^1, \hat{x}^2, \cdots,\hat{x}^d$ be $d$ affinely-independent points from $\Po_0$ satisfying $\alpha^\top x = \beta$. So clearly $\genfrac(){0pt}{1}{\hat{x}^i}{\mathbf{e}_0}$ satisfies \eqref{eq:star0} as an equation, for $1\leq i \leq d$.
For  $i \in N$, let $\tilde{x}^i$ be 
an optimal solution of the lifting problem:
 $M_i := \min \{ \beta -\alpha^\top x ~:~ x \in \Po_i\}$.
It is easy to see that $\genfrac(){0pt}{1}{\tilde{x}^i}{\mathbf{e}_i}$ satisfies \eqref{eq:star0} as an equation.
So, overall, we have $n+d$ points satisfying \eqref{eq:star0} as an equation. In fact, these $n+d$ points satisfy the same properties as the ones
using the same notation in the first case of of the proof of
Theorem \ref{lem:fulldim} (just omitting here the point
$\genfrac(){0pt}{1}{\hat{x}^0}{\mathbf{e}_0}$),
so these points are also affinely independent. 

\noindent\underline{Case 2}: Next, we assume instead that $\Po_k$ is full dimensional for some $k\in N$. Let 
\eqref{eq:star}
be a facet-describing inequality of $\Po_k$\,. Let \eqref{eq:stark}
be its full lifting, where $M_0 := \min \{ \beta -\alpha^\top x ~:~ x \in \Po_0\}$.
Let $\hat{x}^1, \hat{x}^2, \cdots,\hat{x}^d$ be $d$ affinely-independent points from $\Po_k$ satisfying $\alpha^\top x = \beta$. So clearly $\genfrac(){0pt}{1}{\hat{x}^i}{\mathbf{e}_k}$ satisfies \eqref{eq:stark} as an equation, for $1\leq i \leq d$. For $i\in N_0 \setminus\{k\}$, let $\tilde{x}^i$ be an optimal solution of the lifting problem:  $M_i := \min \{ \beta -\alpha^\top x ~:~ x \in \Po_i\}$.
It is easy to see that $\genfrac(){0pt}{1}{\tilde{x}^i}{\mathbf{e}_i}$ satisfies \eqref{eq:stark} as an equation. 
So, overall, we have $n+d$ points satisfying \eqref{eq:stark} as an equation. In fact, these $n+d$ points satisfy the same properties as the ones using the same notation in the second case of  the proof of Theorem \ref{lem:fulldim} (just omitting here the point $\genfrac(){0pt}{1}{\hat{x}^0}{\mathbf{e}_k}$), so these points are also affinely independent. 
\end{proof}

\begin{proof}~\!\!\!\!{\it [of Theorem \ref{lem:nonneg}].}
We need to choose $n+d$ affinely-independent points from $\D$ satisfying $z_j \geq 0$ as an equation to demonstrate that the $z_j\geq 0$ describes a facet of $\D$. Because $\Po_0$ is full dimensional, then $\Po_0$ contains a set of $d+1$ affinely-independent points, say $\hat{x}^0,\hat{x}^1,...,\hat{x}^{d}$.  
Clearly, $
\genfrac(){0pt}{1}{\hat{x}^0}{\mathbf{e}_0}
,\ldots,
\genfrac(){0pt}{1}{\hat{x}^{d}}{\mathbf{e}_0}$ satisfy $z_j \geq 0$ as an equation.
Additionally, because each $\Po_i$ is nonempty for $i \in \{1,...,n\}$, it contains a point, say $\tilde{x}^i$. 
Now consider $n-1$ points (in $\mathbb{R}^{n+d})$ from $\bar{\Po}_i$ for $i \in N \setminus\{j\}$, say $
\genfrac(){0pt}{1}{\tilde{x}^1}{\mathbf{e}_1}
,
\genfrac(){0pt}{1}{\tilde{x}^2}{\mathbf{e}_2}
,\ldots,
\genfrac(){0pt}{1}{\tilde{x}^{j-1}}{\mathbf{e}_{j-1}}
,
\genfrac(){0pt}{1}{\tilde{x}^{j+1}}{\mathbf{e}_{j+1}}
,\ldots,
\genfrac(){0pt}{1}{\tilde{x}^n}{\mathbf{e}_n}.
$
It is obvious that those $n-1$ points satisfy $z_j \geq 0$ as an equation. So overall we have $n+d$ points satisfy $z_j \geq 0$ as an equation. We want to prove that those $n+d$ points are affinely independent. Because these $n+d$ points satisfy the same properties as the ones
using the same notation in the first case of the proof of
Theorem \ref{lem:fulldim} (just omitting here the point
$\genfrac(){0pt}{1}{\tilde{x}^j}{\mathbf{e}_j}$), these points are also affinely independent. 
\end{proof}

\begin{proof}~\!\!\!\!{\it [of Theorem \ref{lem:sumz}].}
We need to choose $n+d$ affinely-independent points from $\D$ satisfying $\sum_{j\in N} z_j \leq 1$ as an equation to demonstrate that the $\sum_{j\in N} z_j \leq 1$ describes a facet of $\D$. Because $\Po_k$ is full dimensional, then $\Po_k$ contains a set of $d+1$ affinely-independent points, say $\hat{x}^0,\hat{x}^1,...,\hat{x}^{d}$.  
Clearly, $
\genfrac(){0pt}{1}{\hat{x}^0}{\mathbf{e}_k}
,\ldots,
\genfrac(){0pt}{1}{\hat{x}^{d}}{\mathbf{e}_k}$ satisfy $\sum_{j\in N} z_j \leq 1$ as an equation.
Additionally, because each $\Po_i$ is nonempty for $i \in N_0\setminus\{k\}$, it contains a point, say $\tilde{x}^i$. We consider $n-1$ points from $\bar{\Po}_i$ for $i \in N \setminus\{k\}$, say$
\genfrac(){0pt}{1}{\tilde{x}^1}{\mathbf{e}_1}
,
\genfrac(){0pt}{1}{\tilde{x}^2}{\mathbf{e}_2}
,\ldots,
\genfrac(){0pt}{1}{\tilde{x}^{k-1}}{\mathbf{e}_{k-1}}
,
\genfrac(){0pt}{1}{\tilde{x}^{k+1}}{\mathbf{e}_{k+1}}
,\ldots,
\genfrac(){0pt}{1}{\tilde{x}^n}{\mathbf{e}_n}.
$
It is obvious that $\genfrac(){0pt}{1}{\tilde{x}^i}{\mathbf{e}_i}, i \in N\setminus\{k\}$ satisfy $\sum_{j\in N} z_j \leq 1$ as an equation. So we have $n+d$ points satisfy $\sum_{j\in N} z_j \leq 1$ as an equation.
In fact, these $n+d$ points satisfy the same properties as the ones
using the same notation in the second case of the proof of
Theorem \ref{lem:fulldim} (just omitting here the point
$\genfrac(){0pt}{1}{\tilde{x}^0}{\mathbf{e}_0}$),
so these points are also affinely independent. 
\end{proof}

\section{Low dimension}\label{sec:lowdim}
Henceforth, we refer to facets of $\D$ described by $z_j\geq 0$ and $\sum_{j\in N} z_j \leq 1$ as
\emph{non-vertical facets}, and other facets of $\D$ as \emph{vertical facets}. 

Next, under the same mild technical condition as in \S\ref{sec:general},
we establish that the facet-describing inequalities identified in Propositions \ref{lem:facet_lift},\ref{lem:nonneg},\ref{lem:sumz}
 give a
complete description of the convex hull $\D$, when $d\in\{1,2\}$.

\begin{thm}\label{thm:hull}
Suppose that  $n\geq 1$, $d\in\{1,2\}$,
all $\Po_i$ are nonempty, for $i\in N_0$\,, and at least one is full dimensional.
Then
the full optimal big M-lifting of
each facet-describing inequality defining  each $\Po_i$\,,
for
$i \in N$, together with the inequalities
$\sum_{i\in N} z_i \leq 1$ and  $z_i \geq 0$, for
$i \in N$,
gives the convex hull $\D$.
\end{thm}   

\begin{proof}\!\!.
By Proposition \ref{lem:fulldim},
$\D$ is a full-dimensional polytope in $\mathbb{R}^{n+d}$, so
every one of its facets 
$\mathcal{F}$
contains $n+d$ affinely independent points of the form
$\genfrac(){0pt}{1}{\hat{x}^i}{\mathbf{e}_i}$, where each such point has $\hat{x}^i$ being an extreme point of a 
$\Po_i$\,. In fact, at most $d+1$ of these can be associated with the same $i$, for each $i\in N_0$\, as the 
different $\hat{x}^i$ for the same $i$ must be affinely independent. 

For $d=1$, we have $n+d=n+1$, 
and one possibility is that the $n+1$ points partition, one for each $i\in N_0$\,. In this case, we can view the 
inequality describing the facet $\mathcal{F}$ as an optimal sequential lifting, starting from an
inequality describing a facet $\mathcal{F}_0$ of $\Po_0$\,. The reason for this is as follows. 
Suppose that
\begin{equation}
\gamma_1 x_1 +\sum_{i\in N} \mu_i z_i \leq \rho
\tag{$*_{F_1}$}\label{starF1}
\end{equation}
describes the facet $\mathcal{F}$, and the $n+1$ points
have the form 
$\genfrac(){0pt}{1}{\hat{x}^i}{\mathbf{e}_i}$,  for $i\in N_0$\,. Plugging $\genfrac(){0pt}{1}{\hat{x}^0}{\mathbf{e}_0}$ into \eqref{starF1} as an equation,
we see that $\rho=\gamma_1 \hat{x}^0$.
Then plugging in $\genfrac(){0pt}{1}{\hat{x}^i}{\mathbf{e}_i}$ for $i\in N$, we see that
$\mu_i=\rho-\gamma_1 \hat{x}^i$, 
 for $i\in N$. So, \eqref{starF1} takes the form
 \begin{equation*}
\gamma_1 x_1 + \gamma_1 \sum_{i\in N} \left( \hat{x}^0 - \hat{x}^i\right) z_i \leq \gamma_1 \hat{x}^0.  
\end{equation*}
Clearly we cannot have $\gamma_1=0$ (the inequality vanishes). For $\gamma_1\not=0$,
the inequality becomes
\begin{equation*}
x_1 +  \sum_{i\in N} \left( \hat{x}^0 - \hat{x}^i\right) z_i ~\left\{\genfrac{}{}{0pt}{1}{\leq}{\geq}\right\}~ \hat{x}^0,  
\end{equation*}
which is a full lifting of  $x \left\{\genfrac{}{}{0pt}{1}{\leq}{\geq}\right\} \hat{x}^0$, 
which must describe a facet of $\Po_0$ or it would be invalid.  

If, instead, the points do not partition like this, then a $\Po_j$ is missed, for some $j\in N_0$. If $j \in N$, then we can view the facet $\mathcal{F}$ as described by $z_j \geq 0$ (corresponding to the missed $\Po_j$).
If, instead, $j=0$, then we can view the facet $\mathcal{F}$ as described by $\sum_{i\in N} z_i \leq 1$.

For $d=2$, we have $n+d=n+2$, 
and one possibility is that the $n+2$ points partition as
two for some $\ell\in N_0$\,, 
and one for each of the remaining $i\in N_0\setminus\{\ell\}$. 
Now, suppose that \begin{equation}
\gamma_1 x_1 + \gamma_2 x_2 + \sum_{i\in N} \mu_i z_i \leq \rho \tag{$*_{F_2}$}\label{starF2}
\end{equation} describes the facet $\mathcal{F}$.
Without loss of generality, 
we can take the $n+2$ points to have the form $\genfrac(){0pt}{1}{\tilde{x}^\ell}{\mathbf{e}_l}$, and  $\genfrac(){0pt}{1}{\hat{x}^i}{\mathbf{e}_i}$  for $i\in N_0$\,. 
 Then, plugging in
 $\genfrac(){0pt}{1}{\tilde{x}^\ell}{\mathbf{e}_l}$ and
 $\genfrac(){0pt}{1}{\hat{x}^\ell}{\mathbf{e}_l}$ into 
 \eqref{starF2} as an equation,
 we have $\gamma_1 \tilde{x}^\ell_1 + \gamma_2 \tilde{x}^\ell_2 +\alpha_\ell z_\ell = \rho$, and
$\gamma_1 \hat{x}^\ell_1 + \gamma_2 \hat{x}^\ell_2 +\alpha_\ell z_\ell = \rho$.
Subtracting,
 we obtain
 $\gamma_1 (\tilde{x}^\ell_1 -  \hat{x}^\ell_1) + \gamma_2 (\tilde{x}^\ell_2 - \hat{x}^\ell_2) = 0$.
Now, consider that 
$\tilde{x}^\ell$ and $\hat{x}^\ell$ are extreme points of a facet (i.e., edge) $\mathcal{F}_\ell$
of $\Po_\ell$\,. 
Suppose that 
$\mathcal{F}_\ell$ is described by $\alpha_1 x_1 +\alpha_2 x_2 \leq \beta$. 
 Then, plugging in 
$\tilde{x}^\ell$ and $\hat{x}^\ell$ into this inequality as an equation,
 we have $\alpha_1 \tilde{x}^\ell_1 + \alpha_2 \tilde{x}^\ell_2 = \beta$, and
 $\alpha_1 \hat{x}^\ell_1 + \alpha_2 \hat{x}^\ell_2 = \beta$.
 Subtracting, we get
 $\alpha_1 (\tilde{x}^\ell_1 -  \hat{x}^\ell_1) + \alpha_2 (\tilde{x}^\ell_2 - \hat{x}^\ell_2) = 0$.
So, we can assume, without loss of generality, that
$\gamma_1=\alpha_1$ and $\gamma_2=\alpha_2$\,.
Finally, validity of \eqref{starF2}
on $\bar{\Po}_0$
implies that $\rho \geq \max\{\alpha_1x_1 + \alpha_2 x_2 ~:~ x \in \Po_0\}=\beta-M_0$\,,
and validity of \eqref{starF2}
on $\bar{\Po}_i$\,, $i\in N$, 
implies that 
$\mu_i \leq \rho - 
\max\{\alpha_1x_1 + \alpha_2 x_2 ~:~ x \in \Po_i\}\leq (\beta-M_0) - (\beta - M_i)= M_i-M_0$\,.
So, if in fact \eqref{starF2} is dominated by 
the full optimal big-M lifting of $\alpha_1 x_1 +\alpha_2 x_2 \leq \beta$,
and it must in fact be
that  full optimal big-M lifting inequality, due to Theorem \ref{lem:facet_lift}.
If, instead, the points do not partition like this, then a $\Po_j$ is missed, for some $j\in N_0$. If $j \in N$, then we can view the facet as described by $z_j \geq 0$ (corresponding to the missed $\Po_j$).
If, instead, $j=0$, then we can view the facet $\mathcal{F}$ as described by $\sum_{i\in N} z_i \leq 1$.
\end{proof}

\begin{rem}
Considering the case of $d=2$ in the sketch of the proof of Theorem \ref{thm:hull}, we can refer to Figure \ref{fig:ex},
where we have $n=1$. We can observe that generally for $d=2$ and $n=1$, facets of $\D$ are either triangles or trapezoids, with a trapezoid (there is one in the example) arising by lifting from a facet of one polytope $\Po_i$ that is parallel to a facet of the other polytope.  
\end{rem}

\begin{rem}\label{rem:falls}
The argument in the proof of Theorem \ref{thm:hull}
falls apart already for $d=3$ and $n=1$. In such a situation,
a vertical facet must contain 4 affinely-independent points
($\D$ is in $\mathbb{R}^4$), but 2 could come from each of 
$\Po_0$ and $\Po_1$\,. In such a case, they do not affinely span a facet of
either (we need 3 affinely independent points on a facet of $\Po_0$ and of $\Po_1$).
So such a vertical facet will not be a lift of a facet-describing inequality 
of $\Po_0$ or $\Po_1$\,.
\end{rem}

In fact, Theorem \ref{thm:hull}
does not extend to $d=3$, as indicated in the 
following result.

\begin{prop}\label{prop:badexample}
For $n=1$ and $d=3$, there can be vertical facets of the 
convex hull $\D$ that are not described by 
 full optimal big-M lifting of facet-describing inequalities defining  $\Po_0$ and $\Po_1$\,, even for a pair of right simplices related by point reflection. 
\end{prop}

\begin{proof}\!\!.
This only requires an example. 
Let 
\[
\Po_0 := \{ x \in \mathbb{R}^3 ~:~  x_1\,,x_2\,,x_3 \leq 5,~ x_1+x_2+x_3 \geq 14
\},
\]
and let 
\[
\Po_1 := \{ x \in \mathbb{R}^3 ~:~  x_1\,,x_2\,,x_3 \geq 0,~ x_1+x_2+x_3 \leq 1 
\}.
\]
Notice that $\Po_0$ and $\Po_1$ are
related by point reflection through the point
$(\frac{5}{2},\frac{5}{2},\frac{5}{2})$. 
It can be checked (using software like \cite{cdd})
that $\D$ has 14 vertical facets,
but there cannot be more than 8 full optimal big-M liftings of facet-describing inequalities, because the total number of facet-describing inequalities
for $\Po_0$ and $\Po_1$ is 16.  
\end{proof}

\begin{rem}\label{rem:badsimplices}
Analyzing the example
in the proof of Proposition
\ref{prop:badexample}
in a bit more detail, there are 6 facets that do not come 
from lifting facet-describing inequalities of the $\Po_i$\,. Namely, those described by
$ 9 -9z_1 \leq x_i+x_j \leq 10-9z_1$\,, for each of the three choices of 
distinct pairs $i,j \in \{1,2,3\}$.
In fact, they can be seen as coming from 
\emph{face}-describing inequalities, namely: $x_i+x_j \leq 1$
and  $x_i+x_j \geq -9$. This is not a good thing!
Solving appropriate linear-optimization problems,
we can check full-dimensionality 
(see \cite{FreundRoundyTodd}) 
and remove redundant inequalities (obvious). 
But we cannot realistically handle (for large $d$) all face-describing inequalities for
each $\Po_i$\,.   

\medskip

It turns out that for this example, the facet-describing inequalities of $\D$ that are not optimal liftings of facet-describing inequalities of the $\Po_i$ are MIR (mixed-integer rounding) inequalities relative to 
weighted combinations of 
optimal liftings of facet-describing inequalities of the $\Po_i$\,.

MIR inequalities are well known
for their high value in tightening MILP formulations; see \cite{GurobiMIR,cornrio,NemhauserWolsey1990,MIR_Closure,ML4MIR}, for example. 
In our notation, 
adapting from \cite[Theorem 8]{cornrio}, 
we
consider a set of the form
\[
S:=\left\{
(x,z)\in \mathbb{R}^d_+\times \{0,1\}^n ~:~
\alpha^\top x + \gamma^\top z \leq \beta
\right\},
\]
and then associated MIR inequality is
\[
\sum_{j\in N} \left(
\lfloor \gamma_j\rfloor + \frac{(f_j-f_0)^+}{1-f_0}
\right) z_j
+
\frac{1}{1-f_0}
\sum_{i ~:~ \alpha_i<0}
\alpha_i x_i \leq \lfloor \beta \rfloor,
\]
where $f_0:=\beta-\lfloor \beta \rfloor$,
and $f_j:=\gamma_j-\lfloor \gamma_j \rfloor$,
for $j\in N$. 

We will demonstrate that,
for the example from the proof of Proposition 
\ref{prop:badexample}, 
all of the facet-describing
inequalities, beyond $0\leq z_1  \leq 1$ 
and optimal full big-M liftings of facet-describing inequalities of the $\Po_i$\,, 
are in fact MIR inequalities.

Starting from the facet-describing inequalities of $\Po_0$\,,
we consider the facet-describing inequalities described by full optimal big-M liftings:
\begin{align}
    x_1 + 4 z_1 &\leq 5 \label{eq:p1}\tag{S.1} \\
    x_2 + 4 z_1 &\leq 5 \label{eq:p2}\tag{S.2} \\
    x_3 + 4 z_1 &\leq 5 \label{eq:p3}\tag{S.3} \\
    -x_1 - x_2 - x_3 - 14 z_1&\leq -14. \label{eq:p4} \tag{T}
\end{align}
Next, we consider the  weighted combinations  $\frac{1}{10}(S.i) + \frac{1}{10}(\ref{eq:p4})$, 
for $i=1,2,3$:
\begin{align*}
-\frac{1}{10} x_2 - \frac{1}{10} x_3 - z_1 & \leq -\frac{9}{10}\\
-\frac{1}{10} x_1 - \frac{1}{10} x_3 - z_1 & \leq - \frac{9}{10}\\
-\frac{1}{10} x_1 - \frac{1}{10} x_2 - z_1 & \leq - \frac{9}{10}.
\end{align*}
Forming the MIR inequalities based on these inequalities and simplifying,
we arrive at
$ 9 -9z_1 \leq x_i+x_j$,
for each pair of distinct $i$ and $j$.

It is easy to check that the mapping
$x_i \rightarrow 5-x_i$ 
is an affine (even unimodular) involution between
$\Po_0$ and $\Po_1$\,;
and extending the mapping to 
$z_1 \rightarrow 1-z_1$\,,
we have an affine  (even unimodular) involution between
$\bar{\Po}_0$ and $\bar{\Po}_1$\,.
So, we can map $\bar{\Po}_1$ to $\bar{\Po}_0$\,,
apply the MIR process exactly as above, and then
map back to $\bar{\Po}_1$\,.
In this way, we derive the
three  inequalities $x_i+x_j \leq 10-9z_1$ (for each pair of distinct $i$ and $j$)
from the MIR process.
\end{rem}


\section{Computing all facets efficiently, in fixed dimension \texorpdfstring{$d$}{d}}\label{sec:allfacets}

We have characterized all facet-describing inequalities of $\D$ for  $d\in\{1,2\}$, and it is easy to see that we can efficiently calculate them. We have also seen that for $d=3$ already,
we do not have a characterization. Nonetheless,
we will see now that for any fixed $d$, we can efficiently calculate all facet-describing inequalities of $\D$.

\begin{thm}\label{thm:computing}
    For $n+1$ polytopes $\Po_i \in \mathbb{R}^d$,  $i \in N_0$\,,
    given by inequality descriptions, we can compute 
    a facet-describing inequality system for $\D$ in time polynomial
    in the input size, considering $d$ to be fixed.
\end{thm}

\begin{proof}\!\!.
Without loss of generality,
we can assume that $d \geq 3$, due to Theorem \ref{thm:hull}. 
Of course we need the non-vertical facets, but there are only $n+1$ of them.

    Suppose now that $\mathcal{F}$ is a vertical facet of $\D$. The facet $\mathcal{F}$ must containing $n+d$ affinely-independent points, and 
    without loss of generality, we 
    can take them to be extreme points of $\D$. We define 
    the \emph{signature} $\sigma=(\sigma_0,\sigma_1,\ldots,\sigma_n)$, where $\sigma_i$ is the number of points coming from $\Po_i$\,. 
    Note that $\sigma$ is not generally uniquely determined by $\mathcal{F}$.
    Because $\mathcal{F}$ is a vertical facet, 
    every $\sigma_i$ is at least one.
    The number of possible 
    distinct such $\sigma$ is exactly $\genfrac(){0pt}{1}{n+d-1}{n}$ (the number of ways to place $n+d-1$ unlabeled balls in $n+1$ labeled urns, with
    at least one ball in each urn).
    For constant $d$, we have 
   $\genfrac(){0pt}{1}{n+d-1}{n} =\mathcal{O}(n^{d-1})$. We can also note that each $\sigma_i \leq d$. 

    Now, if $\sigma_i$ points come from $\Po_i$\,, then their affine span intersects
    $\Po_i$ on an 
    $(\sigma_i-1)$-dimensional face, and we can enumerate all such possible $(\sigma_i-1)$-dimensional faces  of each $\Po_i$ in time
    polynomial in the encoding of $\Po_i$\,. 
    If $\bm{A}_i$ has $m_i$ rows, then a crude upper bound on the number of 
    $(\sigma_i-1)$-dimensional faces  is $\genfrac(){0pt}{1}{m_i}{d-\sigma_i+1}$, which
    is polynomial in $m_i$ for
    fixed $d$, and there are many easy ways to make this constructive. 
    For example, we can first enumerate the set $\mathcal{X}_i$ of extreme points of each $\Po_i$\,, for $i\in N_0$\,. We also enumerate all of the 
    possible signatures $\sigma$. 
    For any given  
    signature $\sigma$, we test
   each selection $S_0\,, S_1\,,\ldots,S_n$ of extreme points of the 
    $\Po_i$ that follows $\sigma$ (i.e., $|S_i|=\sigma_i$). 
    We check if the $n+d$ extreme points $\genfrac(){0pt}{1}{\hat x^i}{\mathbf{e}_i}\in \mathcal{X}_i$\,,
    $i\in N_0$\,.
    If they are, we compute the equation of the hyperplane
    that describes their affine span. Then
    we just check that all of the extreme points of the 
    points $\genfrac(){0pt}{1}{\hat x^i}{\mathbf{e}_i}$, with $\hat x^i\in E_i$\,, $i\in N_0$\, are in one of the two closed half spaces bounded by the hyperplane.
\end{proof}

  \begin{rem}
  Referring to the proof,
  the facets $\mathcal{F}$ (of $\D$) that 
  correspond to optimal full big-M liftings of facet-describing 
  inequalities of the $\Po_i$
  correspond to $\sigma$
  having one component, say $\sigma_j$\,, equal
  to $d$ and the other $n$ components equal to 1.
  There are only $n$ such signatures $\sigma$, but each of them leads to a facet of $\D$ for each facet of the 
  associated $\Po_j$\,.
  \end{rem}


\section{A standard simplex and its reflection}\label{sec:standard}
Next, we take up a generalization of the 
polytopes from the proof of Proposition \ref{prop:badexample}.
For real $a>0$ and $b$, and integer $d \geq 3$,
let 
\[
\Po_0 := \{ x \in \mathbb{R}^d ~:~  x_i \leq b,~ \sum_{i=1}^d x_i \geq db-a\},
\]
and 
\[
\Po_1 := \{ x \in \mathbb{R}^d ~:~  x_i \geq 0,~ \sum_{i=1}^d
x_i \leq a
\}.
\]
Notice that for $d=3$, $a=1$, $b=5$, we get precisely the 
polytopes from the proof of Proposition \ref{prop:badexample} and the ensuing discussion. 
In particular, we observed that
for this case,  the facet-describing inequalities of $\D$ that are not optimal liftings of facet-describing inequalities of the $\Po_i$ are MIR (mixed-integer rounding) inequalities relative to 
weighted combinations of 
optimal liftings of facet-describing inequalities of the $\Po_i$\,. In what follows, 
we will demonstrate that even in the general case, this is true.

First, we need to determine the  the facet-describing inequalities of $\D$ that are not optimal liftings of facet-describing inequalities of the $\Po_i$\,.
Our technique is inspired by the
ideas in \S\ref{sec:allfacets}.

\begin{thm}\label{thm:allfacetssimplex}
    For $d\geq 3$ and $a > 0$, let 
$\Po_0 := \{ x \in \mathbb{R}^d ~:~  x_i \leq b,~ \sum_{i=1}^d x_i \geq db-a\}$
and 
$\Po_1 := \{ x \in \mathbb{R}^d ~:~  x_i \geq 0,~ \sum_{i=1}^d x_i \leq a\}$.    
Then the full optimal big M-lifting of
each facet-describing inequality defining each 
of $\Po_0$ and $\Po_1$\,
together with the inequalities
$0\leq z_1 \leq 1$,
and non-lifting inequalities 
$-\sum_{i \in T_0} x_{i} + (a - (d+1-m) b)z_1 \leq a-(d+1-m)b$ and $\sum_{i \in T_1} x_{i} + \left(  mb-a \right)z_1 \leq mb$, where $T_0\,,T_1 \subseteq \{1, \cdots, d\}$, $|T_0| = d+1-m$, $|T_1| = m$,
and $2\leq m \leq d-1$,
gives the convex hull $\D$.
\end{thm}

\begin{proof}\!\!.
With $a>0$, it is easy to see that 
$\Po_0$ and $\Po_1$ are full dimensional.
Suppose that 
$\mathcal{F}$ is a vertical facet 
of $\D$, but
\emph{not} a facet described by 
an optimal big M-lifting of
a facet-describing inequality for one of the  $\Po_i$\,. 

The polytope $\D$ is in $\mathbb{R}^{d+1}$, so our hypotheses imply that  $\mathcal{F}$
must have $(d+1-m,m)$ as a signature, where $ 2\leq m \leq d-1$; note that (i) signature $(d,0)$ corresponds to $z_1=0$, (ii) signature $(0,d)$ corresponds to $z_1=1$, (iii) signature $(d,1)$
corresponds to a lifting of a facet-describing inequality of $\Po_0$\,, and 
(iii) signature $(1,d)$
corresponds to a lifting of a facet-describing inequality of $\Po_1$\,.
Therefore, for some $m$ satisfying, $ 2\leq m \leq d-1$,
$\mathcal{F}$ must contain 
$d+1-m$  points of $\D$, of the form
$\genfrac(){0pt}{1}{\hat{x}}{0}$,
where $\hat x$ is an extreme point of $\Po_0$
and $m$  points of $\D$, of the form
$\genfrac(){0pt}{1}{\hat{x}}{1}$
where $\hat x$ is an extreme point of $\Po_1$\,.

The extreme points from $\bar{\Po}_0$ are 
\begin{equation*} 
\genfrac(){0pt}{1}{b\mathbf{e}}{0} \mbox{ and }
\genfrac(){0pt}{1}{b\mathbf{e}-a\mathbf{e}_i}{0}
\mbox{ for } i=1,\ldots,d,
\end{equation*} 
and the extreme points from $\bar{\Po}_1$ are 
\begin{equation*} 
\genfrac(){0pt}{1}{\mathbf{e}_0}{1} \mbox{ and }
\genfrac(){0pt}{1}{a\mathbf{e}_i}{1}
\mbox{ for } i=1,\ldots,d.
\end{equation*} 

Accounting for the symmetries,
we can choose from $\bar{\Po}_0$ either 
$\genfrac(){0pt}{1}{b\mathbf{e}}{0}$ and 
$d-m$ distinct points of the form 
$\genfrac(){0pt}{1}{b\mathbf{e}-a\mathbf{e}_i}{0}$, or simply 
$d+1-m$ distinct points of the form 
$\genfrac(){0pt}{1}{b\mathbf{e}-a\mathbf{e}_i}{0}$.
Likewise, 
accounting for the symmetries,
we can choose from $\bar{\Po}_1$ either 
$\genfrac(){0pt}{1}{\mathbf{e}_0}{1}$ and 
$m-1$ distinct points of the form 
$\genfrac(){0pt}{1}{a\mathbf{e}_i}{1}$, or simply 
$m$ distinct points of the form 
$\genfrac(){0pt}{1}{a\mathbf{e}_i}{1}$.
Now, we cannot choose the points from
$\bar{\Po}_0$ and $\bar{\Po}_1$ independently, as the following claims indicate.

\smallskip 

\noindent \underline{Claim 1}:
Suppose that $T_0$ and $T_1$ are
subsets of $\{1,\ldots,d\}$ with 
$|T_0\cap T_1|\geq 2$. Then $\left\{\genfrac(){0pt}{1}{b\mathbf{e}-a\mathbf{e}_i}{0}~:~ i\in T_0\right\}
\cup
\left\{
\vphantom{\genfrac(){0pt}{1}{b\mathbf{e}-a\mathbf{e}_i}{0}}
\genfrac(){0pt}{1}{a\mathbf{e}_i}{1} ~:~ i\in T_1
\right\}
$
is affinely dependent. 

\begin{proof}~\!\!\!\!{\it [of Claim 1].}
It is enough to see that
\[
\genfrac(){0pt}{1}{b\mathbf{e}-a\mathbf{e}_{i_1}}{0}
-
\genfrac(){0pt}{1}{b\mathbf{e}-a\mathbf{e}_{i_2}}{0}
= 
\genfrac(){0pt}{1}{a\mathbf{e}_{i_2}}{1}
-
\genfrac(){0pt}{1}{a\mathbf{e}_{i_1}}{1},
\]
for $T_0\cap T_1 \supset \{i_1\,,\,i_2\}$,
$i_1\not= i_2$\,.
\end{proof}

\noindent \underline{Claim 2}:
Suppose that $T_0$ and $T_1$ are
subsets of $\{1,\ldots,d\}$, with 
$|T_0\cap T_1|\geq 1$. 
Then
$
\left\{
\genfrac(){0pt}{1}{b\mathbf{e}}{0},
\genfrac(){0pt}{1}{\mathbf{e}_0}{1}
\right\}
\cup
\left\{\genfrac(){0pt}{1}{b\mathbf{e}-a\mathbf{e}_i}{0}~:~ i\in T_0\right\}
\cup
\left\{
\vphantom{\genfrac(){0pt}{1}{b\mathbf{e}-a\mathbf{e}_i}{0}}
\genfrac(){0pt}{1}{a\mathbf{e}_i}{1} ~:~ i\in T_1
\right\}
$
is affinely dependent. 

\begin{proof}~\!\!\!\!{\it [of Claim 2].}
It is enough to see that
\[
\genfrac(){0pt}{1}{b\mathbf{e}}{0}
-
\genfrac(){0pt}{1}{b\mathbf{e}-a\mathbf{e}_\ell}{0}
=
\genfrac(){0pt}{1}{a\mathbf{e}_\ell}{1}
-
\genfrac(){0pt}{1}{\mathbf{e}_0}{1},
\]
for any $\ell\in T_0\cap T_1$\,. 
\end{proof}

From these claims, we conclude that there are 
six cases to consider; see Figure \ref{fig:hyper}. We note that for Case 3,
$|T_0|=d-m$, $|T_1|= m-1$, 
and then 
$k$ is defined to be the unique element
in $\{1,\ldots,d\}\setminus (T_0\cup T_1)$. For Case 6, we note that 
the hyperplane equation is independent of
the element $\ell$ in $T_0\cap T_1$\,. 

\begin{figure}
\centering
\resizebox{\textwidth}{!}{
\begin{tabular}{c|c|c|c|c|c}
Case & $\genfrac(){0pt}{1}{b\mathbf{e}}{0}$   & $\genfrac(){0pt}{1}{\mathbf{e}_0}{1}$ & $T_0\cap T_1$ & $ (|T_0|, |T_1|)$ & hyperplane equation\\[3pt]
\hline
1 &\ding{51} & \ding{55} & $\emptyset$ & $(d-m, m)$ & $\sum_{i \in T_1} x_{i} + \left(  mb-a \right)z_1 = mb$ \\
2 &\ding{55} & \ding{51} & $\emptyset$ & $(d+1-m, m-1)$ &  $-\sum_{i \in T_0} x_{i} + (a - (d+1-m) b)z_1 = a-(d+1-m)b$\\
3 &\ding{51} & \ding{51} & $\emptyset$ & $(d-m, m-1)$ & $ x_{k} + b z_1 = b$\\
4 &\ding{51} & \ding{55} & $\{\ell\}$ & $(d-m, m)$ & $ x_{\ell} + b z_1 = b$\\
5 &\ding{55} & \ding{51} & $\{\ell\}$ & $(d+1-m, m-1)$ & $ x_{\ell} + b z_1 = b$\\
6 &\ding{55} & \ding{55} & $\{\ell\}$ & $(d+1-m, m)$ & $-\sum_{i=1}^d x_i + (2a-db)z_1 = a-db$
\end{tabular}
}
\caption{Cases for hyperplanes}\label{fig:hyper}
\end{figure}

In the end, we see that there are four different types of hyperplanes to consider. 
We demonstrate that the two of them  are not facet determining because there is a point in $\D$ on one side and there is another point in $\D$ on the other side:

\smallskip

\noindent \underline{Cases 3, 4, 5}:  $ x_l + bz = b$.
By plugging the point $\genfrac(){0pt}{1}{b\mathbf{e}-a\mathbf{e}_l}{0} \in \D$ into the equation of the hyperplane, the point satisfies the inequality $b-a + b \times 0 < b$. Similarly, substituting the point $\genfrac(){0pt}{1}{a\mathbf{e}_l}{1} \in \D$ yields  $a + 1 \times b > b$.

\smallskip

\noindent \underline{Case 6}: 
$-\sum_{i=1}^d x_i + (2a-db)z = a-db$. By plugging the point $\genfrac(){0pt}{1}{b\mathbf{e}}{0} \in \D$ into the equation of the hyperplane, the point satisfies the inequality $-db + 0 < a-db$. Similarly, substituting the point $\genfrac(){0pt}{1}{\mathbf{e}_0}{1} \in \D$ yields $-0 + 2a-db > a-db$.

\smallskip
Finally, for Cases 1 and 2,
replacing ``$=$'' 
with  ``$\leq$'' in the hyperplane equations, the resulting inequalities
are facet-describing for $\D$
as all of its extreme points satisfy these inequalities.
\end{proof}

\begin{thm}\label{thm:MIRsimplex}
For $d\geq 3$ and $a>0$,
 let 
$\Po_0 := \{ x \in \mathbb{R}^d ~:~  x_i \leq b,~ \sum_{i=1}^d x_i \geq db-a\}$
and 
$\Po_1 := \{ x \in \mathbb{R}^d ~:~  x_i \geq 0,~ \sum_{i=1}^d x_i \leq a\}$.    
Via MIR, starting from other facet-describing inequalities for $\D$, we can obtain the facet-describing inequalities 
$-\sum_{i \in T_0} x_{i} + (a - (d+1-m) b)z_1 \leq a-(d+1-m)b$ and $\sum_{i \in T_1} x_{i} + \left(  mb-a \right)z_1 \leq mb$, where $T_0\,,T_1 \subseteq \{1, \cdots, d\}$, $|T_0| = d+1-m$, $|T_1| = m$, $2\leq m \leq d-1$.
\end{thm}

\begin{proof}\!\!.
We begin by focusing on 
 the facet $\mathcal{F}$ of $\D$ described by
$-\sum_{i \in T_0} x_{i} + (a - (d+1-m) b)z_1 \leq a-(d+1-m)b$, where $T_0\subseteq \{1, \cdots, d\}$, $|T_0| = d+1-m$, 
and $2\leq m \leq d-1$. Such a facet has $(d-m+1,m)$ as a signature.
We will 
show that the inequality describing $\mathcal{F}$ is an MIR inequality, relative to the optimal big-M liftings of facet-describing inequalities of $\Po_i$\,.

\smallskip

\noindent\underline{Case 1}:
$a < (d+1-m)b$.
We sum the lifting inequalities
$x_{i} + (b - a) z_1 \leq b$
over $i\notin T_0$ together with the
lifting inequality 
$ -\sum_{i=1}^d x_i + (a - db)z_1 \leq a-db$, and then divide by 
$(d+1-m)b + (m-2)a$ to obtain
\begin{align*}
    -\sum_{i \in T_0} \frac{1}{(d+1-m)b + (m-2)a} x_i - z_1 \leq \frac{(m-1)a}{(d+1-m)b + (m-2)a} - 1.
\end{align*}
Because $a < (d+1-m)b$, 
we have $\left\lfloor \frac{(m-1)a}{(d+1-m)b + (m-2)a} - 1 \right\rfloor = -1$. Then, $f_0 = \frac{(m-1)a}{(d+1-m)b + (m-2)a} - 1 - \left\lfloor \frac{(m-1)a}{(d+1-m)b + (m-2)a} - 1 \right\rfloor 
= \frac{(m-1)a}{(d+1-m)b + (m-2)a}$. Because the coefficient of $z_1$ is integer, we do not transform the coefficient of $z_1$ when carrying out the MIR procedure. The resulting MIR inequality is 
\[
     - \frac{1}{1-\frac{(m-1)a}{(d+1-m)b + (m-2)a}} \sum_{i \in T_0} \frac{1}{(d+1-m)b + (m-2)a} x_i - z_1 \leq - 1,
\]
which is equivalent to
\[
    - \sum_{i \in T_0} x_i + (a-(d+1-m)b)z_1 \leq a - (d+1-m)b. 
\]

\smallskip

\noindent\underline{Case 2}:
$a > (d+1-m)b$.
We sum the lifting inequalities
$-x_i + (a-b)z_1 \leq a-b$
over $i \in T_0$, and then divide by 
$r := \frac{a(d+1-m)-b(d+1-m) + (a - (d+1-m) b)}{2}$ to obtain
\begin{align*}
    -\sum_{i \in T_0} \frac{1}{r} x_{i} + \frac{a(d+1-m)-b(d+1-m)}{r} z_1 \leq \frac{a(d+1-m)-b(d+1-m)}{r}.
\end{align*}
We have $\left\lfloor \frac{a(d+1-m)-b(d+1-m)}{r} \right\rfloor = \left\lfloor \frac{2}{1 + \frac{a - (d+1-m) b}{a(d+1-m)-b(d+1-m)}} \right\rfloor = 1$ because $0 < a(d-m)$ and $0 < a - (d+1-m) b < a(d+1-m)-b(d+1-m)$. 
And $1 - f_0 = 1 - \left(\frac{2(a(d+1-m)-b(d+1-m))}{a(d+1-m)-b(d+1-m) + (a - (d+1-m) b)} - 1\right) 
= \frac{2(a - (d+1-m) b)}{a(d+1-m)-b(d+1-m) + (a - (d+1-m) b)}$.
We also have $f_0 = f_j$\,, so the coefficient of $z_1$ is $1$. The MIR inequality is 
\[
    -\frac{a(d+1-m)-b(d+1-m) + (a - (d+1-m) b)}{2(a - (d+1-m) b)} \sum_{i \in T_0} \frac{1}{r} x_{i} + z_1 \leq 1 
\]
which is equivalent to
\[
- \sum_{i \in T_0}  x_{i} + (a - (d+1-m) b)z_1 \leq a - (d+1-m) b.
\]

\smallskip

\noindent\underline{Case 3}:
$a = (d+1-m)b$.
When $a = (d+1-m)b$, the inequality
$-\sum_{i \in T_0} x_{i} + (a - (d+1-m) b)z_1 \leq a-(d+1-m)b$
simplifies to $-\sum_{i \in T_0} x_{i} \leq 0$. 

We sum the lifting inequalities
$-x_i + (a-b)z_1 \leq a-b$
over $i\in T_0$ together with the
inequality 
$  ((d-m)a) \left( -z \leq 0 \right)$, and then divide by 
any $r$ such that $r > (d-m)a$ to obtain
\begin{align*}
    -\sum_{i \in T_0} \frac{1}{r(1-f_0)} x_i \leq \left\lfloor \frac{(d-m)a}{r} \right\rfloor = 0,
\end{align*}
which is clearly equivalent to
\[
-\sum_{i \in T_0}  x_i \leq 0.
\]

So, we have established that 
after employing our lifting  inequalities for $\D$, we can obtain the facet-describing inequalities 
$-\sum_{i \in T_0} x_{i} + (a - (d+1-m) b)z_1 \leq a-(d+1-m)b$
via the MIR process.

Finally, we move our attention to
the facet-describing inequalities
$\sum_{i \in T_1} x_{i} + \left(  mb-a \right)z_1 \leq mb$. 
Consider the following three types of lifting inequalities:
\begin{equation*}
    \left\{ 
    \begin{array}{ll}
        x_i + (b - a) z_1 \leq b, & \text{ from } \Po_0\,;\\
        -x_i + (a-b)z_1 \leq a-b,~ \sum_{i=1}^d x_i + (db-a)z_1 \leq db, &  \text{ from } \Po_1\,. 
    \end{array}
    \right.
\end{equation*}
Applying the affine involution $x_i \rightarrow b-x'_i$ for $i=1,\ldots,d$, and 
$z_1 \rightarrow 1-z'_1$\,,  we have 
\begin{equation*}
    \left\{ 
    \begin{array}{l}
        -x'_{i} + (a-b) z'_1 \leq a-b; \\
        x'_{j} + (b - a) z'_1 \leq b,~  -\sum_{i=1}^d x'_i + (a - db)z'_1 \leq a-db.
    \end{array}
    \right.
\end{equation*}
Then, with the same starting inequalities, we just repeat the procedures before, and we will obtain 
$ - \sum_{i \in T_0} x'_i + (a-(d+1-m)b)z' \leq a - (d+1-m)b$. We apply the affine transformation back, $ - \sum_{i \in T_0} (b-x_i) + (a-(d+1-m)b)(1-z) \leq a - (d+1-m)b \Rightarrow \sum_{i \in T_0} x_i + ((d+1-m)b - a)z \leq (d+1-m)b$. We just observe that $T_0$ has the same definition as $T_1$ and we can convert the $T_0$ to $T_1$, and finally we have $\sum_{i \in T_1} x_i + (mb- a)z \leq mb$.
\end{proof}


\section{Common constraint matrix}\label{sec:common}

Finally, under a significant but broad technical condition, 
we will see that 
when the polytopes  have a common 
facet-describing constraint matrix 
(for arbitrary $d\geq 1$ and $n\geq 1$),
 the facet-describing inequalities identified above 
 (in Propositions \ref{lem:facet_lift},\ref{lem:nonneg},\ref{lem:sumz})
 give a
complete description of the convex hull $\D$.

We need to set some 
careful (and somewhat non-standard) notation 
for some standard concepts
to make our theorem and proof precise.
For a full column-rank $\bm{A}\in\mathbb{R}^{m\times d}$, a \emph{basic partition} is a pair $\tau:=(\tau_1,\,\ldots,\tau_d)$
and $\eta:=(\eta_1,\,\ldots,\eta_{m-d})$,
so that every element of the
row indices $\{1,\ldots,m\}$ 
is in precisely one position
of these two ordered sets,
and $\bm{A}_{\tau\cdot}$ is invertible.
For $\bm{b}\in\mathbb{R}^m$ 
and $\bm{c}\in \mathbb{R}^d$ ,
the \emph{primal basic solution} 
associated to $\tau,\eta$ is $\bar{x}:=\bm{A}_{\tau\cdot}^{-1}\bm{b}_\tau \in \mathbb{R}^d$,
and the \emph{dual basic solution} associated to  $\tau,\eta$ is $\bar{y}\in\mathbb{R}^m$
defined by $\bar{y}_\tau^\top:=
\bm{c}^\top \bm{A}_{\tau\cdot}^{-1}$ and
$\bar{y}_\eta:=0$.
 Note that $\bar{x}$ is feasible for 
 $\max \{\bm{c}^\top x:  \bm{A} x \leq \bm{b}\}$ precisely when 
 $\bm{A}_{\eta\cdot}\bar{x} \leq \bm{b}_\eta$\,,
 and $\bar{y}$
 is feasible for the dual
 $\min\{ y^\top \bm{b} ~:~ y^\top \bm{A}=\bm{c},~ y\geq 0\}$ precisely when 
 $\bar{y}_\tau\geq 0$.
It is well known and celebrated that when an optimal solution exists
for $\max \{\bm{c}^\top x:  \bm{A} x \leq \bm{b}\}$,
then there is a basic partition so that the associated primal basic solution is optimal for 
$\max \{\bm{c}^\top x:  \bm{A} x \leq \bm{b}\}$ and
the associated dual basic solution is optimal for
 $\min\{ y^\top \bm{b} ~:~ y^\top \bm{A}=\bm{c},~ y\geq 0\}$. 

\begin{thm}\label{thm:hull3}
For $n \geq 1$ and 
$d \geq 1$, 
let $\Po_i:=\{ x\in\mathbb{R}^d ~:~ \bm{A} x \leq \bm{b}^i\}$
be full dimensional, where 
$\bm{b}^i\in\mathbb{R}^m$, for  $i\in N_0$\,.
We assume that the (single) matrix $\bm{A}$ has full column rank, and that for every row $\bm{A}_{j\cdot}$ of $\bm{A}$,
the inequality $\bm{A}_{j\cdot} x \leq \bm{b}^i_k$ describes a 
nonempty face of every $\Po_i$ and a 
facet of some $\Po_i$\,.
We further suppose:
\begin{equation}
\tag{$\mathrm{\Phi}$}\label{Phi}
\begin{array}{l}
\mbox{For every
basic partition $\tau,\eta$ of the \emph{row} indices of the matrix $\bm{A}$,}\\
\mbox{we have that if }\bm{A}_{\eta\cdot} \bm{A}_{\tau \cdot}^{-1} \bm{b}^i_\tau \leq \bm{b}^i_\eta\,,
\mbox{ holds for some $i$ in $N_0$\,, then it}\\
\mbox{holds for every $i$ in $N_0$\,.}
\end{array}
\end{equation}
Then
the optimal full big-M lifting inequalities 
$\bm{A} x + \sum_{i \in N} (\bm{b}^0 - \bm{b}^i) z_i \leq \bm{b}^0$\,, 
together with $\sum_{j\in N} z_j \leq 1$, and 
  $z_j\geq 0, \mbox{ for } j\in N$,
gives a minimal inequality description of the convex hull $\D$.
\end{thm} 

\begin{rem} We wish to emphasize
that in Theorem \ref{thm:hull3},
the dependence of the $\Po_i$ on $i$ is only through 
$\bm{b}^i$\,. 
Of course we already know that the result is true for $d=1,2$ by Theorem \ref{thm:hull}. So the further contribution of Theorem \ref{thm:hull3} is for $d\geq 3$. 
\end{rem}

\begin{rem}
We note that 
Theorem \ref{thm:hull3} is
closely related to the main result of \cite{Jeroslow}. But 
the model and results of \cite{Jeroslow} 
do not employ binary variables to model the disjunction.
\end{rem}

Before taking up the proof of Theorem \ref{thm:hull3}, we state a very relevant corollary, a remark, and a few examples, which will shine 
some light on some key aspects of the theorem and its proof.

From the Theorem \ref{thm:hull3}, we can derive the following concrete
special case of hyper-rectangles,  which is interesting and useful in many applications. 

\begin{cor}[{\cite[Theorem 7]{QL_ISCO2024}}]
\label{cor:hyper}
For $n \geq 1$ and 
$d \geq 1$, we consider  $n+1$ hyper-rectangles 
$\Po_j:= [\bm{\ell}_{j1},\bm{u}_{j1}] \times \cdots \times [\bm{\ell}_{jd},\bm{u}_{jd}]$,
for $j\in N_0$\,.
The full  optimal big-M lifting inequalities 
$x_i + \textstyle  \sum_{j \in N} (\bm{\ell}_{0i} - \bm{\ell}_{ji}) z_j \geq \bm{\ell}_{0i}$\,, 
$\textstyle x_i + \sum_{j \in N} (\bm{u}_{0i} - \bm{u}_{ji}) z_j \leq \bm{u}_{0i}\,, \text{ for all } i = 1, \cdots, d$,
together with $\sum_{j\in N} z_j \leq 1$, and 
  $z_j\geq 0, \mbox{ for } j\in N$,
gives the convex hull $\D$.  
\end{cor}

\begin{proof}\!\!.
We simply observe that here we have that
for \emph{every}
basic partition $\tau,\eta$ of the row indices of the matrix $\bm{A}=[I_d ~|~ -I_d]^\top$,
we have 
$
\bm{A}_{\eta\cdot} \bm{A}_{\tau\cdot}^{-1} \bm{b}^i_{\tau\cdot} \leq \bm{b}^i_\eta\,,
$
for \emph{all} $i\in N_0$\,.
That is, the
basic solution 
corresponding to every basic partition is feasible  for each  $\Po_i$\,.
\end{proof}

\begin{rem}
Concerning the hypothesis:
``for every row $\bm{A}_{j\cdot}$ of $\bm{A}$,
the inequality $\bm{A}_{j\cdot} x \leq \bm{b}^i_k$ describes a 
nonempty face of every $\Po_i$ and a 
facet of some $\Po_i$'',  
it is easy to check the following two statements:
If $\bm{A}_{j\cdot} x \leq \bm{b}^i_k$
describes an empty face of $\Po_i$ for some $i\in N_0$\,, then all of the $n$ optimal full big-M liftings of it  will be redundant for $\D$. 
If $\bm{A}_{j\cdot} x \leq \bm{b}^i_k$
does not describe a facet of $\Po_i$ for all $i\in N_0$\,, then none of its $n$ optimal full big-M liftings of it  will 
describe a facet of $\D$. 
\end{rem}

\begin{ex}
Concerning the hypothesis \ref{Phi}, 
we refer to $\Po_0$ and $\Po_1$ from the proof of Proposition
\ref{prop:badexample},
but we incorporate some redundant inequalities in such a way that we satisfy the hypothesis 
``for every row $\bm{A}_{j\cdot}$ of $\bm{A}$,
the inequality $\bm{A}_{j\cdot} x \leq \bm{b}^i_k$ describes a 
nonempty face of every $\Po_i$ and a 
facet of some $\Po_i$''.
Specifically, we let
\[
\bm{A}:=
\begin{bmatrix}[r]
    1 & 0 & 1 \\
    0 & 1 & 0 \\
    0 & 0 & 1 \\
    -1&-1&-1 \\
    -1 & 0 & 0 \\
    0 & -1 & 0 \\
    0 & 0 & -1 \\
    1 & 1 & 1 
\end{bmatrix},
~\bm{b}^0 :=
\begin{bmatrix}[r]
 5 \\ 5 \\ 5 \\ -14 \\ -4 \\ -4 \\ -4 \\ 15
\end{bmatrix},
~\bm{b}^1 :=
\begin{bmatrix}[r]
1 \\ 1 \\ 1 \\ 0 \\ 0 \\ 0 \\ 0 \\ 1
\end{bmatrix}.
\]
In this description, we shift each of the four facet describing inequalities of each of $\Po_0$ and $\Po_1$ so that it describes a nonempty face (actually a 0-dimensional face) of the other.

As we indicated in Remark \ref{rem:badsimplices}, 
there are 6 facets that do not come 
from lifting facet-describing inequalities of the $\Po_i$\,. Namely, those described by
$ 9 -9z_1 \leq x_i+x_j \leq 10-9z_1$\,, for each of the three choices of 
distinct pairs $i,j \in \{1,2,3\}$. Therefore, the conclusion of Theorem 
\ref{thm:hull3} fails.

But we can see that choosing, for example,
$\tau=(1,2,3)$ and $\eta=(4,5,7,8)$, we obtain the basic solution $\bar{x}=(5,5,5)^\top$ 
which is indeed in $\Po_0$\,, but
$\bar{x}=(1,1,1)^\top$
which is \emph{not} in $\Po_1$\,. So the 
hypothesis \ref{Phi} fails.
\end{ex}

\begin{ex}
We give a construction to show that there are examples of other polytopes satisfying the hypotheses of Theorem \ref{thm:hull3} besides hyper-rectangles. Our construction applies in any dimension $d\geq 2$, but in light of Theorem \ref{thm:hull}, the relevancy is for $d\geq 3$. 
In a sense, our construction extracts 
and generalizes the aspects of 
hyper-rectangles that we
exploited.
We begin with any system $\bm{A}x\leq \bm{b}$, giving an irredundant description of
a \emph{simple} full-dimensional polytope $\Po_0$ in $\mathbb{R}^d$.
Now, we construct 
$\Po_1$ (the same can sequentially work for more polytopes) by changing the right-hand sides, one by one, so that the 
combinatorics of the polytope does not change. There is always a small enough change so that this is the case, but large changes can work too.
See, for example, Figure \ref{fig:shift}, with one polytope in black and the other in red.
\end{ex}

\begin{figure}[ht]
\centering
\includegraphics[width=0.65\textwidth]{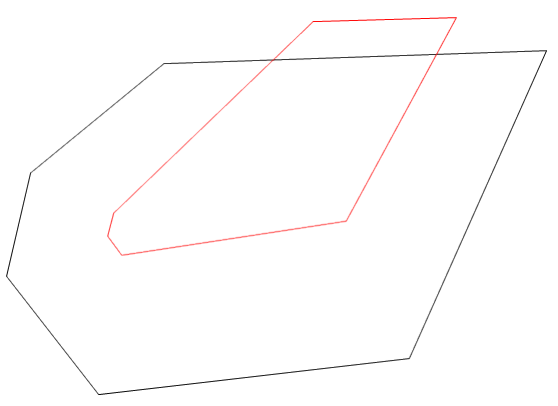}
\caption{An example that is not 
 a set of hyper-rectangles:}\label{fig:shift}
\end{figure}

\begin{proof}~\!\!\!\!{\it [of  Theorem  \ref{thm:hull3}}.]
Consider the associated linear-optimization problem with \emph{arbitrary} $\bm{c} \in \mathbb{R}^d$, $\bm{g} \in \mathbb{R}^n$:
\[
\begin{array}{lll}
& \max~ \bm{c}^{\top}x + \bm{g}^{\top}z\\
&\quad \mbox{subject to}  \\[2pt]
&\textstyle \bm{A} x + \sum_{i \in N} M^{i,0} z_i \leq \bm{b}^0; \\[2pt]
&\textstyle \bm{A} x + \sum_{i \in N \setminus \{k\}}(M^{i,k} - M^{0,k})z_i - M^{0,k} z_k \leq \bm{b}^k - M^{0,k}, \text{ for all } k \in N; \\[2pt]
&\textstyle \mathbf{e}^{\top}z \leq 1;\\
&z\geq 0,
\end{array}
\]
where (letting $m$ denote the number of rows of $\bm{A}$)
\begin{align*}
    M^{i,j}_l :=&~ \min \{\bm{b}^j_l - \bm{A}_{l \cdot}x : x \in \Po_i\}, ~\forall~ l = 1,...,m\,,~ i\in N_0\,,~ j\in N_0 \\
     =&~\, \bm{b}^j_l - \max \{ \bm{A}_{l \cdot}x : x \in \Po_i\}\\
     =&~\, \bm{b}^j_l - \bm{b}^i_l\,.~~ \mbox{ (because
     $\bm{A}_{l \cdot}x \leq \bm{b}^i_l$ describes a nonempty face of $\Po_i$).}
\end{align*} 
Therefore, 
$ M^{i,j} = \bm{b}^j - \bm{b}^i$. 

 Now, it is easy to check that \emph{there
is a unique  full optimal big-M lifting
corresponding to each row of $\bm{A}$,} and we can simplify the linear-optimization problem to:
\[
\begin{array}{lll}
\tag{${P}_H$}\label{primalH}
& \max~ \bm{c}^{\top}x + \bm{g}^{\top}z & {\color{red}\mbox{dual}}\\
&\quad \mbox{subject to}  & {\color{red}\mbox{\underline{var}}}\\[2pt]
&\textstyle \bm{A} x + \sum_{i \in N} (\bm{b}^0 - \bm{b}^i) z_i \leq \bm{b}^0;~ &  {\color{red} y \in \mathbb{R}^{m}}\\[2pt]
&\textstyle \mathbf{e}^{\top}z \leq 1; & {\color{red} \pi \in \mathbb{R}} \\
&z\geq 0.
\end{array}
\]

Our goal is to 
give a recipe that solves the linear-optimization problem \ref{primalH}\,, using always an extreme point of $\D$.
If so, then we will have shown that the inequalities of  \ref{primalH} are the facet-describing inequalities for $\D$.

For $k \in N_0$\,, Let $\hat{x}^{k} := \argmax \{\bm{c}^\top x:  \bm{A} x \leq \bm{b}^k\}$. So, $\hat x ^k$ maximizes $\bm{c}^\top x$ on $\Po_k$\,. 
Now, let  $\hat k
:= \argmax_{k\in N_0} \{ \bm{c}^\top \hat{x}^k + \bm{g}^\top\mathbf{e}_k \}$.
Therefore, 
$\genfrac(){0pt}{1}{\hat x^{\hat k}}{\mathbf{e}_{\hat k}}$
maximizes $\bm{c}^{\top}x + \bm{g}^{\top}z$ on $\D$.
We can immediately observe that 
$\genfrac(){0pt}{1}{\hat x^{\hat k}}{\mathbf{e}_{\hat k}}$
is feasible for \ref{primalH}\,,
which is a relaxation of 
$\D$.   Therefore, to prove that 
$\genfrac(){0pt}{1}{\hat x^{\hat k}}{\mathbf{e}_{\hat k}}$
is optimal for \ref{primalH}\,,
we consider its dual:
\[
\begin{array}{lll}
\tag{${D}_H$}\label{dualH}
& \min~ y^{\top} \bm{b}^0 + \pi \\
&\quad \mbox{subject to}\\
& y^{\top} \bm{A}= \bm{c}^{\top} ; \\
& y^{\top} (\bm{b}^0 - \bm{b}^i) + \pi \geq \bm{g}_i\, , \text{ for all } i \in N;\\
& y \geq 0 ;~ \pi \geq 0.
\end{array}
\]
Next, we construct a dual solution 
$\genfrac(){0pt}{1}{\hat y}{\hat \pi}$
of \ref{dualH}\,, defined by
\begin{align*}
& \hat{y} :=  \argmin\{y^\top \bm{b}^{0} ~:~ 
 y^\top \bm{A} = \bm{c}^{\top},~  y \geq 0\};\\
& \hat{\pi}  :=\begin{cases} 
    0, & \text{ if } \hat{k}=0;\\
    \hat{y}^\top \bm{b}^{\hat{k}} - \hat{y}^\top \bm{b}^0 + \bm{g}_{\hat{k}}\,, & \text{ if } \hat{k} \in N.
  \end{cases}
\end{align*}
It remains only to check that this solution is feasible for  \ref{dualH} and has the same objective value in \ref{dualH}
as $\genfrac(){0pt}{1}{\hat x^{\hat k}}{\mathbf{e}_{\hat k}}$
has in \ref{primalH}\,. 
Then, by weak duality, it follows that $\genfrac(){0pt}{1}{\hat x^{\hat k}}{\mathbf{e}_{\hat k}}$
is optimal for \ref{primalH}\,.
Because we find always an extreme point of $\D$ to solve
\ref{primalH}\,, we can conclude that the $\D$ is precisely the feasible region of \ref{primalH}\,.

We will check
now that this dual solution is feasible for \ref{dualH}.

\noindent $\bullet$ By the construction of $\hat{y}$, it is immediate that $\hat{y} \geq 0$ and that $\hat{y}^\top \bm{A}= \bm{c}$.

\noindent $\bullet$ Next, we check that $\hat{\pi} \geq 0$. If $\hat{k} = 0$, then $\hat{\pi} = 0$, so we can 
assume that $\hat{k} \in N$.
Note that 
we have constructed $\hat{y}$
so that $\bm{c}^\top\hat{x}^0=
\hat{y}^\top\bm{b}^0$, by strong duality applied to $\max \{\bm{c}^\top x:  \bm{A} x \leq \bm{b}^0\}$.
Now, because of property \ref{Phi}, the same basic 
partition that is primal and dual optimal for 
$\max \{\bm{c}^\top x:  \bm{A} x \leq \bm{b}^0\}$
must also be feasible for 
$\max \{\bm{c}^\top x:  \bm{A} x \leq \bm{b}^k\}$, for all $k\in N$.
But it is trivially feasible
for the dual of
all of the $\max \{\bm{c}^\top x:  \bm{A} x \leq \bm{b}^k\}$, because they
all have the same constraints
as the dual of $\max \{\bm{c}^\top x:  \bm{A} x \leq \bm{b}^0\}$.
Therefore, 
$\hat{y}^{\top} \bm{b}^{\hat{k}} = \bm{c}^\top {\hat x}^{\hat k}$\,.
Now,   we have 
$
    \hat{\pi}  = \hat{y}^{\top} \bm{b}^{\hat{k}} - \hat{y}^{ \top} \bm{b}^0 + \bm{g}_{\hat{k}} 
    = \bm{c}^\top {\hat x}^{\hat k}  +  \bm{g}^\top {\mathbf{e}}_{\hat k} - \bm{c}^\top {\hat x}^0, 
$
which is nonnegative due to the choice of $\hat k$.  So we have $\hat \pi\geq 0$. 

\noindent $\bullet$ Next, we will check that
$\hat{y}^{ \top} (\bm{b}^0 - \bm{b}^i) + \pi \geq \bm{g}_i\, , \text{ for all } i \in N$.

If $\hat{k} = 0$, we have 
\begin{align*}
& \hat{y}^{ \top} (\bm{b}^0 - \bm{b}^i) + \pi \geq \bm{g}_i\,, ~ \forall i \in N\\
\iff & 0 \geq \hat{y}^{\top} (\bm{b}^{i} - \bm{b}^0) + \bm{g}_{i} \\
\iff & \hat{y}^{\top} \bm{b}^0 \geq \hat{y}^{\top} \bm{b}^{i} + \bm{g}_{i} \\
\iff &  \bm{c}^\top {\hat x}^0 \geq \bm{c}^\top {\hat x}^{i}  +  \bm{g}^\top {\mathbf{e}}_{i} ~~(\text{because }  \hat{y}^{\top} \bm{b}^{i} = \bm{c}^\top {\hat x}^{i}),
\end{align*}
which is true by the choice of $\hat{k}=0$.

Otherwise, we have $\hat{k} \in N$, and
\begin{align*}
& \hat{y}^{\top} (\bm{b}^0 - \bm{b}^i) + \pi \geq \bm{g}_i\,, ~ \forall i \in N\\
\iff & \hat{y}^{\top} (\bm{b}^0 - \bm{b}^i) + \hat{y}^{\top} (\bm{b}^{\hat{k}} - \bm{b}^0) + \bm{g}_{\hat{k}} \geq \bm{g}_i \\
\iff & \hat{y}^{\top} \bm{b}^{\hat{k}} + \bm{g}_{\hat{k}} \geq \hat{y}^{\top}\bm{b}^i + \bm{g}_i \\
\iff &  \bm{c}^\top {\hat x}^{\hat k}  +  \bm{g}^\top {\mathbf{e}}_{\hat k} \geq \bm{c}^\top {\hat x}^i + \bm{g}^\top {\mathbf{e}}_{i}\,,
\end{align*}
which is again true by the choice of $\hat{k}$.

\medskip

\noindent Therefore, 
$\genfrac(){0pt}{1}{\hat y}{\hat \pi}$
is feasible for \ref{dualH} for $k \in N_0$\,.

Finally, we check that the objective value of  $\genfrac(){0pt}{1}{\hat x^{\hat k}}{\mathbf{e}_{\bar k}}$ in \ref{primalH} and 
$\genfrac(){0pt}{1}{\hat y}{\hat \pi}$
in \ref{dualH} are the same. 
If $\hat k=0$, we have
\begin{align*}
    \bm{c}^\top \hat{x}^{\hat{k}} + \bm{g}^\top \mathbf{e}_0 =  \bm{c}^\top \hat{x}^{\hat{k}} = \hat{y}^\top \bm{b}^{\hat{k}} + 0.
\end{align*}
Otherwise, $\hat{k} \in N$, and 
\begin{align*}
    & \bm{c}^\top {\hat x}^{\hat k}  +  \bm{g}^\top {\mathbf{e}}_{\hat k} \\
    & \quad= \hat{y}^{\top} \bm{b}^{\hat{k}} + \bm{g}_{\hat{k}} \\
    & \quad= \hat{y}^{\top} \bm{b}^0 + \hat{y}^{\top} (\bm{b}^{\hat{k}} - \bm{b}^0) + \bm{g}_{\hat{k}} \\
    & \quad = \hat{y}^{\top} \bm{b}^0 + \hat{\pi}.
\end{align*}
Therefore, we have equality of primal and dual optimal values, and the result follows. 
\end{proof}

\section{Outlook}\label{sec:outlook}
We would like to see how far we can go with generalizing Theorems 
\ref{thm:allfacetssimplex} and \ref{thm:MIRsimplex}, beyond the particular families of polytopes that we considered. In particular, it is quite possible, that $\D$ 
could be completely described by applying one round of MIR inequalities after 
first employing the full optimal big-M lifting inequalities.
Resolving this in the positive for $d=3$ would be a nice result, and we hope that the ideas in the proofs of Theorems 
\ref{thm:allfacetssimplex} and \ref{thm:MIRsimplex} can motivate such an investigation.  

We are also interested in 
developing a broad sufficient conditions (broader than
afforded by Theorem \ref{thm:hull3}) under which
Theorem \ref{thm:hull} would extend to $d=3$. 

Finally, we plan to carry out computational experiments to substantiate the practical value of our results.

\section*{Acknowledgments} The work of both authors was supported in part by Office of Naval Research grants N00014-21-1-2135 and N00014-24-1-2694. 
The work of J. Lee was additionally supported by
the Gaspard Monge Visiting Professor Program at \'Ecole Polytechnique (Palaiseau), and also by the 
National Science Foundation under grant DMS-1929284, while he was in residence at the Institute for Computational and Experimental Research in 
Mathematics (ICERM) at Providence, RI, during the Discrete Optimization semester program, 2023.

\bibliographystyle{alpha}
{\footnotesize
\bibliography{lowdimdisj}}

\newcommand{\etalchar}[1]{$^{#1}$}
\begin{thebibliography}{GDKK24}

\bibitem[AHS07]{audet}
Charles Audet, Jean Haddad, and Gilles Savard.
\newblock Disjunctive cuts for continuous linear bilevel programming.
\newblock {\em Optimization Letters}, 1(3):259--267, 2007.
\newblock \url{https://doi.org/10.1007/s11590-006-0024-3}.

\bibitem[Bal79]{BalasDP}
Egon Balas.
\newblock Disjunctive programming.
\newblock In P.~Hammer, E.~Johnson, and B.~Korte, editors, {\em Annals of
  Discrete Mathematics 5: Discrete Optimization}, pages 3--51. North Holland,
  1979.
\newblock \url{https://doi.org/10.1016/S0167-5060(08)70342-X}.

\bibitem[Bal88]{BALAS1988279}
Egon Balas.
\newblock On the convex hull of the union of certain polyhedra.
\newblock {\em Operations Research Letters}, 7(6):279--283, 1988.
\newblock \url{https://doi.org/10.1016/0167-6377(88)90058-2}.

\bibitem[Bal18]{balas2018disjunctive}
Egon Balas.
\newblock {\em Disjunctive Programming}.
\newblock Springer, 2018.
\newblock \url{https://doi.org/10.1007/978-3-030-00148-3}.

\bibitem[Bla90]{Blair}
Charles Blair.
\newblock Representation for multiple right-hand sides.
\newblock {\em Mathematical Programming}, 49:1--5, 1990.
\newblock \url{https://doi.org/10.1007/BF01588775}.

\bibitem[BLL{\etalchar{+}}11]{belotti2011disjunctive}
Pietro Belotti, Leo Liberti, Andrea Lodi, Giacomo Nannicini, and Andrea
  Tramontani.
\newblock Disjunctive inequalities: applications and extensions.
\newblock {\em Wiley Encyclopedia of Operations Research and Management
  Science}, 2:1441--1450, 2011.
\newblock \url{https://doi.org/10.1002/9780470400531.eorms0537}.

\bibitem[CCZ14]{CCZ_Book}
Michele Conforti, Gérard Cornuéjols, and Giacomo Zambelli.
\newblock {\em Integer Programming}.
\newblock Springer, 2014.
\newblock \url{https://doi.org/10.1007/978-3-319-11008-0}.

\bibitem[Cor08]{cornrio}
Gérard Cornuéjols.
\newblock Valid inequalities for mixed integer linear programs.
\newblock {\em Mathematical Programming}, 112:3--44, 03 2008.
\newblock \url{https://doi.org/10.1007/s10107-006-0086-0}.

\bibitem[Dan63]{Dantzig}
George~B. Dantzig.
\newblock {\em Linear Programming and Extensions}.
\newblock RAND Corporation, Santa Monica, CA, 1963.
\newblock \url{https://www.rand.org/pubs/reports/R366.html}.

\bibitem[DGL07]{MIR_Closure}
Sanjeeb Dash, Oktay G\"{u}nl\"{u}k, and Andrea Lodi.
\newblock On the {MIR} closure of polyhedra.
\newblock In {\em Proceedings of the 12th International Conference on Integer
  Programming and Combinatorial Optimization}, IPCO '07, pages 337--351,
  Berlin, Heidelberg, 2007. Springer.
\newblock \url{https://doi.org/10.1007/978-3-540-72792-7_26}.

\bibitem[FRT85]{FreundRoundyTodd}
Robert~M. Freund, Robin Roundy, and Michael~J. Todd.
\newblock Identifying the set of always-active constraints in a system of
  linear inequalities by a single linear program, 1985.
\newblock Sloan W.P. No. 1674-85 (Rev). Available at:
  \url{https://dspace.mit.edu/handle/1721.1/2111}.

\bibitem[Fuk22]{cdd}
Komei Fukuda.
\newblock cdd, 2022.
\newblock \url{https://github.com/cddlib/cddlib}.

\bibitem[GDKK24]{ML4MIR}
Oscar Guaje, Arnaud Deza, Aleksandr~M. Kazachkov, and Elias~B. Khalil.
\newblock Machine learning for optimization-based separation: the case of
  mixed-integer rounding cuts, 2024.
\newblock \url{https://arxiv.org/abs/2408.08449}.

\bibitem[Gr{\"u}03]{Grunbaum2003}
Branko Gr{\"u}nbaum.
\newblock {\em Convex Polytopes (2nd edition prepared by Volker Kaibel, Victor
  Klee, and G{\"u}nter M. Ziegler)}.
\newblock Springer New York, New York, NY, 2003.
\newblock \url{https://doi.org/10.1007/978-1-4613-0019-9}.

\bibitem[{Gur}22]{GurobiM}
{Gurobi~Optimization}.
\newblock Dealing with big-{M} constraints, 2022.
\newblock
  \url{https://www.gurobi.com/documentation/9.5/refman/dealing_with_big_m_constra.html}~,
  September 12, 2022.

\bibitem[{Gur}24]{GurobiMIR}
{Gurobi~Optimization}.
\newblock {MIRCuts}, 2024.
\newblock
  \url{https://www.gurobi.com/documentation/current/refman/mircuts.html}~,
  October 27, 2024.

\bibitem[Jer88]{Jeroslow}
Robert~G. Jeroslow.
\newblock A simplification for some disjunctive formulations.
\newblock {\em European Journal of Operational Research}, 36(1):116--121, 1988.
\newblock \url{https://doi.org/10.1016/0377-2217(88)90013-6}.

\bibitem[KKY15]{kilincc2015two}
Fatma K{\i}l{\i}n{\c{c}}-Karzan and Sercan Y{\i}ld{\i}z.
\newblock Two-term disjunctions on the second-order cone.
\newblock {\em Mathematical Programming}, 154(1):463--491, 2015.
\newblock \url{https://doi.org/10.1007/s10107-015-0903-4}.

\bibitem[KM21]{Misener_Disj}
Jan Kronqvist and Ruth Misener.
\newblock A disjunctive cut strengthening technique for convex {MINLP}.
\newblock {\em Optimization and Engineering}, 22:1315--1345, 09 2021.
\newblock \url{https://doi.org/10.1007/s11081-020-09551-6}.

\bibitem[Lee24]{LeeLP}
Jon Lee.
\newblock {\em A First Course in Linear Optimization (4th Ed., v4.08)}.
\newblock Reex Press, 2013--24.
\newblock \url{https://github.com/jon77lee/JLee_LinearOptimizationBook}.

\bibitem[Lee04]{Lee_Cambridge}
Jon Lee.
\newblock {\em A First Course in Combinatorial Optimization}.
\newblock Cambridge University Press, 2004.
\newblock \url{https://doi.org/10.1017/CBO9780511616655}.

\bibitem[LSS22]{lee_gaining_2020}
Jon Lee, Daphne Skipper, and Emily Speakman.
\newblock Gaining or losing perspective.
\newblock {\em Journal of Global Optimization}, 82:835--862, 2022.
\newblock \url{https://doi.org/10.1007/s10898-021-01055-6}.

\bibitem[LSSX21]{PLPerspecCTW}
Jon Lee, Daphne Skipper, Emily Speakman, and Luze Xu.
\newblock Gaining or losing perspective for piecewise-linear under-estimators
  of convex univariate functions.
\newblock In {\em Graphs and Combinatorial Optimization, CTW 2020. {\rm AIRO
  Springer Series, Vol. 5. C. Gentile, G. Stecca, P. Ventura, eds.}}, pages
  349--360, 2021.
\newblock \url{https://doi.org/10.1007/978-3-030-63072-0_27}.

\bibitem[LSSX23]{lee2020piecewise}
Jon Lee, Daphne Skipper, Emily Speakman, and Luze Xu.
\newblock Gaining or losing perspective for piecewise-linear under-estimators
  of convex univariate functions.
\newblock {\em Journal of Optimization Theory and Applications}, 196:1--35,
  2023.
\newblock \url{https://doi.org/10.1007/s10957-022-02144-6}.

\bibitem[LTV23]{Lodi-et-al:2020}
Andrea Lodi, Mathieu Tanneau, and Juan-Pablo Vielma.
\newblock Disjunctive cuts in mixed-integer conic optimization.
\newblock {\em Mathematical Programming}, 199:671--719, 2023.
\newblock \url{https://doi.org/10.1007/s10107-022-01844-1}.

\bibitem[NW88]{NWbook}
George~L. Nemhauser and Laurence~A. Wolsey.
\newblock {\em Integer and Combinatorial Optimization}.
\newblock Wiley, 1988.
\newblock \url{https://doi.org/10.1002/9781118627372}.

\bibitem[NW90]{NemhauserWolsey1990}
George~L. Nemhauser and Laurence~A. Wolsey.
\newblock A recursive procedure to generate all cuts for 0–1 mixed integer
  programs.
\newblock {\em Mathematical Programming}, 46:379--390, 1990.
\newblock \url{https://doi.org/10.1007/BF01585752}.

\bibitem[QL24]{QL_ISCO2024}
Yushan Qu and Jon Lee.
\newblock On disjunction convex hulls by lifting.
\newblock In {\em Combinatorial Optimization: {\rm Proceedings of the 8th
  International Symposium, ISCO 2024, La Laguna, Tenerife, Spain, May 22–24,
  2024. Lecture Notes in Computer Science, vol 14594}}, pages 3--15. Springer,
  2024.
\newblock \url{https://doi.org/10.1007/978-3-031-60924-4_1}.

\bibitem[TG15]{TRESPALACIOS201598}
Francisco Trespalacios and Ignacio~E. Grossmann.
\newblock Improved {Big-M} reformulation for generalized disjunctive programs.
\newblock {\em Computers \& Chemical Engineering}, 76:98--103, 2015.
\newblock \url{https://doi.org/10.1016/j.compchemeng.2015.02.013}.

\bibitem[Vie15]{Vielma2015MixedIL}
Juan~Pablo Vielma.
\newblock Mixed integer linear programming formulation techniques.
\newblock {\em SIAM Review}, 57:3--57, 2015.
\newblock \url{https://doi.org/10.1137/130915303}.

\bibitem[Zie95]{ziegler}
Günter~M. Ziegler.
\newblock {\em {Lectures on Polytopes, {\rm Graduate Texts in Mathematics,
  152}}}.
\newblock Springer-Verlag, 1995.
\newblock \url{https://doi.org/10.1007/978-1-4613-8431-1}.

\end{thebibliography}

\end{document}